\documentclass[11pt]{amsart}

\usepackage{amssymb,graphicx}
\usepackage{verbatim,esint}
\usepackage{color}

\DeclareMathOperator*{\osc}{osc}

\usepackage[colorlinks=true,urlcolor=blue,
citecolor=red,linkcolor=blue,linktocpage,pdfpagelabels,bookmarksnumbered,bookmarksopen]{hyperref}
\usepackage[hyperpageref]{backref}

\newtheorem{teo}{Theorem}[section]
\newtheorem{lm}[teo]{Lemma}
\newtheorem{prop}[teo]{Proposition}
\newtheorem{coro}[teo]{Corollary}
\newtheorem*{main}{Main Theorem}
\theoremstyle{definition}

\newtheorem{oss}[teo]{Remark}
\newtheorem*{ack}{Acknowledgements}

\numberwithin{equation}{section}

\author[Bousquet]{Pierre Bousquet}
\address[P. Bousquet]{Institut de Math\'ematiques de Toulouse, CNRS UMR 5219
\newline\indent Universit\'e de Toulouse
\newline\indent F-31062 Toulouse Cedex 9, France.}
\email{pierre.bousquet@math.univ-toulouse.fr}

\author[Brasco]{Lorenzo Brasco}
\address[L.\ Brasco]{Dipartimento di Matematica e Informatica
\newline\indent
Universit\`a degli Studi di Ferrara
\newline\indent
Via Machiavelli 35
\newline\indent
44121 Ferrara, Italy}
\address{{\it and } Institut de Math\'ematiques de Marseille
\newline\indent
Aix-Marseille Universit\'e,
Marseille, France}
\email{lorenzo.brasco@unife.it}

\title[Orthotropic $p-$harmonic functions in the plane]{$C^1$ regularity of\\ orthotropic $p-$harmonic functions\\ in the plane}

\date{\today}

\begin{document}

\begin{abstract}
We prove that local weak solutions of the orthotropic $p-$harmonic equation in $\mathbb{R}^2$ are $C^1$ functions.
\end{abstract}

\maketitle

\begin{center}
\begin{minipage}{11cm}
\small
\tableofcontents
\end{minipage}
\end{center}

\section{Introduction}

\subsection{The result}
Let $\Omega\subset \mathbb{R}^2$ be an open set and let \(u\in W^{1,p}_{\rm loc}(\Omega)\) be a local weak solution of the {\it orthotropic $p-$Laplace equation}
\begin{equation}
\label{eq-Euler-classic}
\sum_{i=1}^{2} \left(|u_{x_i}|^{p-2}\,u_{x_i}\right)_{x_i}=0.
\end{equation}
This means that for every  $\Omega'\Subset \Omega$ and every \(\varphi\in W^{1,p}_0(\Omega')\), we have
\begin{equation}
\label{eq-Euler-basic}
\sum_{i=1}^{2} \int_{\Omega'}|u_{x_i}|^{p-2}\,u_{x_i}\, \varphi_{x_i}\,dx=0.
\end{equation}
In the recent literature, such an equation has sometimes been called the {\it pseudo $p-$Laplace equation}. 
We decided to adopt the terminology {\it orthotropic $p-$Laplace equation} in order to emphasize the role played by the coordinate system. Indeed, let us recall that if \(u\in W^{1,p}_{\rm loc}(\Omega)\) is a local weak solution of the usual {\it $p-$Laplace equation}, i.e.
\[
\sum_{i=1}^{2} \left(|\nabla u|^{p-2}\,u_{x_i}\right)_{x_i}=0,
\] 
then for every linear isometry \(A:\mathbb{R}^2\to \mathbb{R}^2\), \(u\circ A\) is still a local weak solution of this equation on \(A^{-1}(\Omega)\). This property fails to be true for equation \eqref{eq-Euler-classic}, but it still holds if $A$ belongs to the dihedral group $\mathrm{D}_2$, i.e. the group of symmetries of the square \((-1,1)\times(-1,1)\).   
\vskip.2cm\noindent
A function \(u\in W^{1,p}_{\rm loc}(\Omega)\) is a local weak solution if and only of it is a local minimizer of the functional 
\[
\mathfrak{F}(\varphi;\Omega'):=\sum_{i=1}^2 \frac{1}{p}\,\int_{\Omega'}|\varphi_{x_i}|^p  \,dx,\qquad \varphi\in W^{1,p}_{\rm loc}(\Omega),\ \Omega'\Subset\Omega\subset\mathbb{R}^2.
\]
This easily follows from the convexity of the functional \(\mathfrak{F}\). We recall that $u\in W^{1,p}_{\rm loc}(\Omega)$ is a local minimizer of $\mathfrak{F}$ if 
\[
\mathfrak{F}(u;\Omega')\le \mathfrak{F}(\varphi;\Omega'),\qquad \mbox{ for every }u-\varphi\in W^{1,p}_0(\Omega'), \ \Omega'\Subset\Omega.
\]
In the recent paper \cite{BBJ}, we proved that for $p\ge 2$ any such local minimizer is a locally Lipschitz function (actually, the case \(1<p<2\) is a mere application of \cite[Theorem 2.2]{FF}).
The aim of this paper is to go one step further and prove the following additional regularity.
\begin{main}
\label{teo:theorem-continuity-N2}
Every local minimizer $U\in W^{1,p}_{\rm loc}(\Omega)$ of the functional $\mathfrak{F}$ is a $C^1$ function.
\end{main}

\begin{oss}
It is easy to see that the function
\[
u(x_1,x_2)=|x_1|^\frac{p}{p-1}-|x_2|^\frac{p}{p-1},\qquad (x_1,x_2)\in\mathbb{R}^2,
\]
is a local weak solution of \eqref{eq-Euler-classic}. Observe that for $p>2$, $u$ is not $C^2$, but only $C^{1,1/(p-1)}$. As in the case of the standard $p-$Laplacian, we conjecture this to be the sharp regularity of local weak solutions.
\end{oss}

\subsection{Method of proof}

The proof of the Main Theorem is greatly inspired by that of \cite[Theorem 11]{Santambrogio-Vespri} by Santambrogio and Vespri, which in turn exploits an idea introduced by DiBenedetto and Vespri in \cite{DiBV}. However, since our equation is much more singular/degenerate than theirs, most of the estimates have to be recast and the argument needs various nontrivial adaptations.
 In order to neatly explain the method of proof and highlight the differences with respect to \cite{Santambrogio-Vespri}, let us first recall their result.
\par
In \cite{Santambrogio-Vespri} it is shown that in $\mathbb{R}^2$, local weak solutions of the variational equation
\begin{equation}
\label{SV}
\mathrm{div} \nabla H(\nabla u)=0,
\end{equation}
are such that $x\mapsto\nabla H(\nabla u(x))$ is continuous, provided that:
\begin{itemize}
\item $\nabla H(\nabla u)\in W^{1,2}_{\rm loc}\cap L^\infty_{\rm loc}$;
\vskip.2cm
\item $H:\mathbb{R}^2\to [0,\infty)$ is a $C^2$ convex function such that there exist $M\ge 0$ and $0<\lambda\le \Lambda$ for which
\begin{equation}
\label{hessienne}
\lambda\,|z|^{p-2}\,|\xi|^2\le\langle D^2 H(z)\,\xi,\xi\rangle\le  \Lambda\,|z|^{p-2}\, |\xi|^2,\qquad \mbox{ for every }\xi\in\mathbb{R}^2,\ |z|\ge M.
\end{equation}
\end{itemize}
The last assumption implies that \eqref{SV} is a degenerate/singular elliptic equation, with {\it confined degeneracy/singularity}. Indeed, on the set where the gradient of a Lipschitz solution $u$ satisfies  $|\nabla u|\ge M$, the equation behaves as a  uniformly elliptic equation. By using the terminology of \cite{BBJ}, we can say that \eqref{SV} has a {\it $p-$Laplacian structure at infinity}.
\par
The proof of the  continuity of $\nabla H(\nabla u)$ in \cite{Santambrogio-Vespri} relies on the following De Giorgi--type lemma: given a ball \(B_R\) of radius $R$, if  a component $H_{x_i}(\nabla u)$ of the vector field \(\nabla H (\nabla u)\) has large oscillations only on a small portion of \(B_R\), then the global oscillation of \(H_{x_i}(\nabla u)\) on the ball \(B_{R/2}\) is reduced (in a precise quantitative sense).  Such a result amounts to an $L^\infty$ estimate for (a nonlinear function of) the gradient, which in turn relies on the Caccioppoli inequality for the linearized equation
\begin{equation}
\label{linearized}
\mathrm{div} \left(D^2(\nabla u)\,\nabla u_{x_i}\right)=0.
\end{equation}
On the contrary, if $H_{x_i}(\nabla u)$ has large oscillations on a large portion of $B_R$, then one exploits the fact  that a function $W^{1,2}\cap L^\infty$ in the plane is such that:
\begin{itemize}
\item[(A1)] either its Dirichlet energy in a crown contained in $B_R$ is large;
\vskip.2cm
\item[(A2)] or the function itself is large on a circle contained in $B_R$. 
\end{itemize}
When (A2) occurs, the structure of the linearized equation \eqref{linearized} allows to prove a {\it minimum principle} for $H_{x_i}(\nabla u)$, which implies that $H_{x_i}(\nabla u)$ is large on the whole disc bounded  by the above mentioned circle. This again leads to a decay of the oscillation of $H_{x_i}(\nabla u)$ (this time because the infimum increases when shrinking the ball).
\par
Then the continuity result of \cite{Santambrogio-Vespri} is achieved by constructing inductively a decreasing sequence of  balls and using the dichotomy above at each step. The important point is that since $H_{x_i}(\nabla u)$ has finite Dirichlet energy, then possibility (A1) {\it can occur only finitely many times}. Hence, the oscillation of $H_{x_i}(\nabla u)$ decays to $0$, as desired.
\vskip.2cm\noindent
Unfortunately, our equation \eqref{eq-Euler-classic} has not a $p-$Laplacian structure at infinity, i.\,e.\,\eqref{hessienne} is not satisfied. Indeed, in our case we have 
\[
H(z)=\sum_{i=1}^2 \frac{|z_i|^p}{p}\quad \mbox{ so that }\quad 
D^2H (z)=(p-1)\,\left[\begin{array}{cc} |z_1|^{p-2} &0\\ 0& |z_2|^{p-2} \end{array}\right],\quad z=(z_1,z_2)\in\mathbb{R}^2.
\]
In particular, \(D^2H(z)\) is degenerate/singular on the union of the two axes $\{z_1=0\}\cup \{z_2=0\}$
and our equation does not fit in the framework of \cite{Santambrogio-Vespri}. Thus, even if the proof of the Main Theorem will follow the guidelines illustrated above, we have to overcome the additional difficulties linked to the more degenerate/singular structure of \eqref{linearized}. In particular, in the case $p>2$, we will need a new Caccioppoli inequality, which weirdly mixes different components of the gradient (see Proposition \ref{prop:mixing}). This is one of the main novelties of the paper.

\begin{oss}[Stream functions]
For $1<p<\infty$, let us set $p'=p/(p-1)$. When $\Omega\subset\mathbb{R}^2$ is simply connected, to every local weak solution $u\in W^{1,p}_{\rm loc}(\Omega)$ of \eqref{eq-Euler-classic} one can associate a {\it stream function} $v\in W^{1,p'}_{\rm loc}(\Omega)$, such that
\[
v_{x_1}=|u_{x_2}|^{p-2}\,u_{x_2}\qquad \mbox{ and }\qquad v_{x_2}=-|u_{x_1}|^{p-2}\,u_{x_1}.
\]
It is readily seen that $v$ is a weak solution of
\[
\sum_{i=1}^{2} \left(|v_{x_i}|^{p'-2}\,v_{x_i}\right)_{x_i}=0.
\]
Existence of such a function $v$ is a straightforward consequence of the {\it Poincar\'e Lemma}, once it is observed that \eqref{eq-Euler-classic} implies that the vector field
\[
\left(|u_{x_1}|^{p-2}\,u_{x_1},|u_{x_2}|^{p-2}\,u_{x_2}\right),
\]
is divergence free (in the distributional sense). This would allow to reduce the proof of the Main Theorem to the case $1<p\le 2$ only. However, this kind of argument is very specific to the homogeneous equation and {\it already fails} in the case
\[
\sum_{i=1}^{2} \left(|u_{x_i}|^{p-2}\,u_{x_i}\right)_{x_i}=\lambda\in\mathbb{R},
\]
which on the contrary is covered by our method (indeed, observe that the previous equation and \eqref{eq-Euler-classic} have the same linearization \eqref{linearized}, thus the Main Theorem still applies). More generally, we observe that our method of proof can be adapted to treat the case (as in \cite{Santambrogio-Vespri}) of
\[
\sum_{i=1}^{2} \left(|u_{x_i}|^{p-2}\,u_{x_i}\right)_{x_i}=f,
\]
under suitable (not sharp) assumptions\footnote{As in the case of the ordinary $p-$Laplacian (see \cite[Corollary 1.6]{KM}), the sharp assumption should be $f\in L^{2,1}_{\rm loc}$, the latter being a Lorentz space. For $p>2$ our proof requires
\[
|u_{x_j}|^\frac{p-2}{2}\,u_{x_j}\in W^{1,2}_{\rm loc}(\Omega),
\]
a result which is true only when $f$ enjoys suitable differentiability properties.} on $f$. For these reasons, we avoided to use this argument based on stream functions.
\end{oss}

\subsection{Plan of the paper}
We first warn the reader that almost every section is divided in two parts, one for the degenerate case $p>2$ and the other for the singular one $1<p < 2$ ( the case \(p=2\) corresponds to the standard Laplacian). Though the methods of  proof for the two cases look very much the same, there are some important differences which lead us to think that it is better to separate the two cases.
\par
In Section \ref{sec:2} we introduce the technical machinery and present some basic integrability properties of  solutions and their derivatives, needed throughout the whole paper. Section \ref{sec:3} is devoted to some new Caccioppoli inequalities for the gradient of a local minimizer. The core of the paper is represented by Sections \ref{sec:4} and \ref{sec:5}, concerning decay estimates for a nonlinear function of the gradient (case $p>2$) or for the gradient itself (case $1<p\le 2$). Finally, the proof of the Main Theorem is postponed to Section \ref{sec:6}. The paper ends with two Appendices containing technical facts.

\begin{ack}
The idea for the weird Caccioppoli inequality of Proposition \ref{prop:mixing} comes from a conversation with Guillaume Carlier in March 2011, we wish to thank him. Peter Lindqvist is gratefully acknowledged for a discussion on stream functions in June 2014.
Part of this work has been written during some visits of the first author to Marseille and Ferrara and of the second author to Toulouse. Hosting institutions and their facilities are kindly acknowledged. 
The second author is a member of the {\it Gruppo Nazionale per l'Analisi Matematica, la Probabilit\`a
e le loro Applicazioni} (GNAMPA) of the {\it Istituto Nazionale di Alta Matematica} (INdAM).
\end{ack}

\section{Preliminaries}
\label{sec:2}

\subsection{Notation}
Given \(\lambda>0\) and a ball \(B\subset \mathbb{R}^2\) of radius \(R>0\), we denote by \(\lambda\,B\) the ball with the same center and radius \(\lambda\,R\). 
\par
We define for every \(q >-1\) the function $g_q:\mathbb{R}\to\mathbb{R}$ as
\begin{equation}
\label{gq}
g_q(t)=|t|^q\, t,\qquad t\in \mathbb{R}.
\end{equation}
Then \(g_q\) is a homeomorphism and \(g_{q}^{-1}= g_{-q/(q+1)}\). Observe that
\[
|t|^{q}\,t\le \alpha\qquad \Longleftrightarrow \qquad t\le |\alpha|^{-\frac{q}{q+1}}\,\alpha,
\]
a fact that will be used repeatedly.
\par
Let $U\in W^{1,p}_{\rm loc}(\Omega)$ be a given local minimizer of $\mathfrak{F}$. We fix a ball \(B\Subset \Omega\). There exists \(\lambda_B>1\) such that $\lambda_B\,B\Subset \Omega$ as well.
If $\{\rho_{\varepsilon}\}_{\varepsilon>0}\subset C^{\infty}_0(B_\varepsilon)$ is a smooth convolution kernel (here, \(B_{\varepsilon}\) refers to the ball with center \(0\) and radius \(\varepsilon\)), we define \(U^\varepsilon:=U\ast\rho_{\varepsilon}\in W^{1,p}(\Omega_{\varepsilon})\) where \(\Omega_{\varepsilon}:=\{x\in \Omega : \mathrm{dist}(x, \partial \Omega)>\varepsilon\}\).
By definition of $U^\varepsilon$ there exists $0<\varepsilon_0<1$ such that for every $0<\varepsilon\le \varepsilon_0$
\begin{equation}
\label{Uepsilon}
\|U^\varepsilon\|_{W^{1,p}(B)}= \|\nabla U^\varepsilon\|_{L^p(B)}+\|U^\varepsilon\|_{L^p(B)}\le \|\nabla U\|_{L^p(\lambda_B\,B)}+\|U\|_{L^p(\lambda_B\,B)}.
\end{equation}

\subsection{Regularization scheme, case $p>2$}

As in \cite{BBJ}, we consider the  minimization problem
\begin{equation}
\label{approximation}
\min\left\{\sum_{i=1}^2\frac{1}{p}\,\int_{B}|w_{x_i}|^p\,dx + \frac{p-1}{2}\,\varepsilon\, \int_{B}|\nabla w|^2 \,dx\, :\, w-U^\varepsilon \in W^{1,p}_0(B)\right\}.
\end{equation}
Since the functional is strictly convex, there exists a unique solution \(u^\varepsilon\), which is smooth on \(\overline{B}\) (see e.g. \cite[Theorem 2.4]{BBJ}). Moreover, \(u^\varepsilon\) satisfies the Euler-Lagrange equation
\[
\sum_{i=1}^2\int_{B} (|u^\varepsilon_{x_i}|^{p-2}+(p-1)\,\varepsilon)\,u^\varepsilon_{x_i}\,\varphi_{x_i}  \,dx=0,\qquad \mbox{ for every }\varphi\in W^{1,p}_0(B).
\]
We take \(\varphi\in C^2\) with compact support in $B$. Then for \(j\in \{1, 2\}\),  the partial derivative \(\varphi_{x_j}\)  is still an admissible test function. An integration by parts leads to
\begin{equation}
\label{Differentiation_Euler_equation_varepsilon}
\sum_{i=1}^2\int_{B}\Big(|u^\varepsilon_{x_i}|^{p-2}+\varepsilon\Big)\,u^\varepsilon_{x_i x_j}\,\varphi_{x_i}  \,dx=0,\qquad j=1,2.
\end{equation}
As usual, by a density argument, the equation still holds with $\varphi \in W^{1,2}_0(B)$. 
We now collect some uniform estimates on $u^\varepsilon$. 
\begin{lm}[Uniform energy estimate]
\label{lm:energy}
There exists a constant $C=C(p)>0$ such that for every $0<\varepsilon<\varepsilon_0$ the following estimate holds 
\begin{equation}
\label{uniformeg}
\int_B |\nabla u^\varepsilon|^p\, dx\le C\left(\int_{\lambda_B\,B} |\nabla U|^p\,dx+\,\varepsilon^\frac{p}{p-2}\,|B|\right).
\end{equation}
Moreover, the family $\{u^{\varepsilon}\}_{0<\varepsilon<\varepsilon_0}$ converges weakly in $W^{1,p}(B)$ and strongly in $L^p(B)$ to $U$.
\end{lm}
\begin{proof}
The estimate \eqref{uniformeg} is standard, it is sufficient to test the minimality of $u^\varepsilon$ against $U^\varepsilon$, which is admissible. In particular, the family $\{u^\varepsilon\}_{0<\varepsilon<\varepsilon_0}$ is uniformly bounded in $W^{1,p}(B)$.
Moreover, by \cite[Lemma 2.9]{BBJ}  there exists a sequence $\{\varepsilon_k\}_{k\in\mathbb{N}}\subset(0,\varepsilon_0)$ such that $u^{\varepsilon_k}$ converges weakly in $W^{1,p}(B)$ and strongly in $L^p(B)$ to a solution $w$ of 
\[
\min\left\{\sum_{i=1}^2\frac{1}{p}\,\int_{B}|\varphi_{x_i}|^p\,dx\, :\, \varphi-U\in W^{1,p}_0(B)\right\}.
\]
Since $U$ is a local minimizer of $\mathfrak{F}$ and the solution of this problem is unique (by strict convexity), we get $w=U$ and full convergence of the whole family.
\end{proof}
\begin{lm}[Uniform regularity estimates]
\label{lm:regularity}
For every $0<\varepsilon<\varepsilon_0$ and every $B_r\Subset B$ we have
\begin{equation}
\label{uniform}
\|u^{\varepsilon}\|_{L^\infty(B_r)}\le C, 
\end{equation}
\begin{equation}
\label{uniform1}
\|\nabla u^{\varepsilon}\|_{L^\infty(B_r)} \le C,
\end{equation}
and
\begin{equation}
\label{uniform2} 
\int_{B_r} \left|\nabla \left(|u^{\varepsilon}_{x_j}|^{\frac{p-2}{2}}\,u^{\varepsilon}_{x_j}\right)\right|^2\,dx \leq C,\qquad j=1,2,
\end{equation}
for some constant $C>0$ independent of $\varepsilon>0$.
\end{lm}
\begin{proof}
The proof of the $L^\infty$ estimate \eqref{uniform} is standard, it can be obtained as in \cite[Chapter 7]{Gi}.
\vskip.2cm\noindent
The Lipschitz estimate \eqref{uniform1} is more delicate and is one of the main outcome of \cite{BBJ}. Indeed, we know from \cite[Proposition 4.1]{BBJ} that there exists $C=C(p)>0$ such that for every $B_r\Subset B_R\Subset B$
\begin{equation}
\label{8?!}
\Big\|u^\varepsilon_{x_i}\Big\|_{L^\infty(B_{r})}\le C\left(\frac{R}{R-r}\right)^8\,\left[\fint_{B_{R}} \left|\nabla u^\varepsilon\right|^p\, dx+1\right]^{2+\frac{1}{p}},\qquad i=1,2,
\end{equation}
With the notation introduced in \cite{BBJ}, this corresponds to the particular case $\delta_1=\delta_2=0$ and $f=0$ there. By combining this with \eqref{uniformeg}, we get \eqref{uniform1}.
\vskip.2cm\noindent
We now prove the $W^{1,2}$ estimate for the nonlinear function of $\nabla u_\varepsilon$. We take $\eta\in C^\infty_0(B)$ a standard cut-off function such that
\[
0\le \eta\le 1,\qquad\eta  \equiv 1 \textrm{ on } B_r, \qquad \eta \equiv 0 \mbox{ on } \mathbb{R}^2\setminus B_{R}, \qquad |\nabla \eta| \leq \frac{C}{R-r}.
\]
Then we test \eqref{Differentiation_Euler_equation_varepsilon} against $\varphi=u^\varepsilon_{x_j}\,\eta^2$. With standard manipulations, we get the Caccioppoli inequality
\[
\sum_{i=1}^2\int \Big(|u^\varepsilon_{x_i}|^{p-2}+\varepsilon\Big)\,\left|u^\varepsilon_{x_i x_j}\right|^2\,\eta^2  \,dx\le C\,\sum_{i=1}^2 \int \Big(|u^\varepsilon_{x_i}|^{p-2}+\varepsilon\Big)\,|u^\varepsilon_{x_j}|^2\,|\eta_{x_i}|^2\,dx.
\]
By dropping the term containing $\varepsilon$ on the left and observing that
\[
|u^{\varepsilon}_{x_i}|^{p-2}\,\left|u^{\varepsilon}_{x_i x_j}\right|^2=\frac{4}{p^2}\,\left|\left(|u^\varepsilon_{x_i}|^\frac{p-2}{2}\,u^\varepsilon_{x_i}\right)_{x_j}\right|^2,
\]
we get
\begin{equation}
\label{sobolevoso}
\sum_{i=1}^2\int_{B_r} \left|\left(|u^\varepsilon_{x_i}|^\frac{p-2}{2}\,u^\varepsilon_{x_i}\right)_{x_j}\right|^2\,dx\le \frac{C}{(R-r)^2}\,\sum_{i=1}^2 \int_{B_R} \Big(|u^\varepsilon_{x_i}|^{p-2}+\varepsilon\Big)\,|u^\varepsilon_{x_j}|^2\,dx,
\end{equation}
where we used the properties of $\eta$. In order to conclude, it is sufficient to use again \eqref{uniformeg}.
\end{proof}

From  the bounds obtained in Lemma \ref{lm:regularity}, we can deduce the following convergence result. 
\begin{prop}[Convergence]
\label{lm_convergence_ueps}
With the notation above,  
for every $B_r\Subset B$ we have:
\vskip.2cm
\begin{enumerate}
\item[{\it i)}] \(\{u^{\varepsilon}\}_{0<\varepsilon<\varepsilon_0}\) converges uniformly to \(U\) on \(\overline{B_r}\);
\vskip.2cm
\item[{\it ii)}] $\left\{|u^{\varepsilon}_{x_i}|^{\frac{p-2}{2}}\,u^{\varepsilon}_{x_i}\right\}_{0<\varepsilon<\varepsilon_0}$ converges to \(|U_{x_i}|^{\frac{p-2}{2}}\,U_{x_i}\) weakly in $W^{1,2}(B_r)$ and strongly in $L^2(B_r)$. In particular, we have
\[
|U_{x_i}|^{\frac{p-2}{2}}\,U_{x_i}\in W^{1,2}(B_r); 
\]
\vskip.2cm
\item[{\it iii)}] \(\{\nabla u^{\varepsilon}\}_{0<\varepsilon<\varepsilon_0}\) converges to \(\nabla U\) strongly in $L^p(B_r)$.
\vskip.2cm
\end{enumerate}
\end{prop}
\begin{proof}
We already know from Lemma \ref{lm:energy} that $u^\varepsilon$ converges to $U$ weakly in $W^{1,p}(B)$ and strongly in $L^p(B)$.
\par
In view of  \eqref{uniform} and \eqref{uniform1}, the Arzel\`a-Ascoli Theorem implies that the convergence is indeed uniform on $\overline B_r$, for every $B_r\Subset B$.
\par
By \eqref{uniform2}, there exists a sequence \(\{\varepsilon_k\}_{k\in \mathbb{N}}\subset (0, \varepsilon_0)\) such that
\[
\left\{|u^{\varepsilon_k}_{x_i}|^{\frac{p-2}{2}}\,u^{\varepsilon_k}_{x_i}\right\}_{k\in \mathbb{N}},\qquad i=1,2,
\]
converges to some function \(V_i\in W^{1,2}(B_r)\), weakly in \(W^{1,2}(B_r)\) and strongly in \(L^{2}(B_r)\). In particular, this is a Cauchy sequence in $L^2(B_r)$. By using the elementary inequality
\[
|t-s|^p\le C\,\Big||t|^\frac{p-2}{2}\,t-|s|^\frac{p-2}{2}\,s\Big|^2,\qquad t,s\in\mathbb{R},
\]
where \(C>0\) depends only on \(p\), 
we obtain that $\{u^{\varepsilon_k}_{x_i}\}_{k\in \mathbb{N}}$  is a Cauchy sequence as well, this time in $L^p(B_r)$. This implies that 
\[
\lim_{k\to +\infty}\left\|\nabla u^{\varepsilon_k}-\nabla U\right\|_{L^p(B_r)}=0.
\]
We now prove that $V_i=|U_{x_i}|^{(p-2)/2}\,U_{x_i}$. We use the elementary inequality
\[
\Big||t|^\frac{p-2}{2}\,t-|s|^\frac{p-2}{2}\,s\Big|\le C\, \left(|t|^\frac{p-2}{2}+|s|^\frac{p-2}{2}\right) |t-s|,\qquad t,s\in\mathbb{R},
\]
valid for some $C=C(p)>0$. Then we obtain
\[
\begin{split}
\int_{B_r} \left||u^{\varepsilon_k}_{x_i}|^{\frac{p-2}{2}}\,u^{\varepsilon_k}_{x_i}-|U_{x_i}|^{\frac{p-2}{2}}\,U_{x_i}\right|^2\,dx&\le C\, \int_{B_r} \left(|u^{\varepsilon_k}_{x_i}|^{\frac{p-2}{2}}+|U_{x_i}|^{\frac{p-2}{2}}\right)^2\,|u^{\varepsilon_k}_{x_i}-U_{x_i}|^2\,dx\\
&\le C\,\left( \int_{B_r} \left(|u^{\varepsilon_k}_{x_i}|^{\frac{p-2}{2}}+|U_{x_i}|^{\frac{p-2}{2}}\right)^\frac{2\,p}{p-2}\,dx\right)^\frac{p-2}{p}\\
&\times \left(\int_{B_r} |u^{\varepsilon_k}_{x_i}-U_{x_i}|^p\,dx\right)^\frac{2}{p}.
\end{split}
\]
By using the strong convergence of the gradients proved above, this implies that \(V_i=|U_{x_i}|^{(p-2)/2}\,U_{x_i}\). Since the above argument can be repeated for every subsequence of \(\{u^\varepsilon\}_{0<\varepsilon<\varepsilon_0}\), it follows from the uniqueness of the limit that the convergence holds true for the whole family \(\{u^\varepsilon\}_{0<\varepsilon<\varepsilon_0}\), both in {\it ii)} and {\it iii)}. The proof is complete.
\end{proof}

From the convergence results stated in the above proposition, we can obtain some regularity properties for the local minimizer $U$ that we state in the following theroem. These properties, which come with local scaling invariant a priori estimates, have already been established in \cite{BBJ}, \cite{BC} and \cite{FF}. 
\begin{teo}[A priori estimates, $p>2$]
\label{teo:apriori}
Every local minimizer $U\in W^{1,p}_{\rm loc}(\Omega)$ of the functional $\mathfrak{F}$ is a locally Lipschitz function, such that for every $\alpha\ge p/2$ we have
\[
|U_{x_i}|^{\alpha-1}\,U_{x_i}\in W^{1,2}_{\rm loc}(\Omega),\qquad i=1,2.
\]  
Moreover, for every $B_R\Subset\Omega$ we have
\begin{equation}
\label{apriori}
\|U_{x_i}\|_{L^\infty(B_{R/2})}\le C_1\,\left(\fint_{B_R} |\nabla U|^p\,dx\right)^\frac{1}{p},\qquad i=1,2,
\end{equation}
\begin{equation}
\label{apriorialpha}
\int_{B_{R/2}} \left|\nabla \left(|U_{x_i}|^{\alpha-1}\,U_{x_i}\right)\right|^2\,dx\le C_2\left(\fint_{B_R} |\nabla U|^p\,dx\right)^\frac{2\,\alpha}{p},\qquad i=1,2,
\end{equation}
for some $C_1=C_1(p)>0$ and $C_2=C_2(p,\alpha)>0$.
\end{teo}
\begin{proof}
Let us prove the estimates \eqref{apriori} and \eqref{apriorialpha}. By taking the limit as $\varepsilon$ goes to $0$ in \eqref{8?!} and using the convergence result of Proposition \ref{lm_convergence_ueps}, we obtain
\[
\Big\|U_{x_i}\Big\|_{L^\infty(B_{R/2})}\le C\,\left[\fint_{B_{R}} \left|\nabla U\right|^p\, dx+1\right]^{2+\frac{1}{p}},\qquad i=1,2.
\]
In order to obtain \eqref{apriori}, it is sufficient to observe that if $U$ is a local minimizer of $\mathfrak{F}$, then for every $\lambda>0$ the function $\lambda\,U$ is still a local minimizer of the same functional. Thus the previous Lipschitz estimate holds true, i.e.
\[
\lambda\,\Big\|U_{x_i}\Big\|_{L^\infty(B_{R/2})}\le C\,\left[\lambda^p\,\fint_{B_{R}} \left|\nabla U\right|^p\, dx+1\right]^{2+\frac{1}{p}},\qquad i=1,2.
\]
This can be rewritten as
\[
\lambda^\frac{p}{2\,p+1}\,\Big\|U_{x_i}\Big\|_{L^\infty(B_{R/2})}^\frac{p}{2\,p+1}-C\,\lambda^p\,\fint_{B_{R}} \left|\nabla U\right|^p\, dx\le C,\qquad i=1,2,
\]
for a different constant $C=C(p)>0$. If we now maximize the left-hand side with respect to $\lambda>0$, we get \eqref{apriori} as desired.
\vskip.2cm\noindent
We already know from Proposition \ref{lm_convergence_ueps} that $|U_{x_i}|^{(p-2)/2}\,U_{x_i} \in W^{1,2}_{\rm loc}(\Omega)$. By passing to the limit in \eqref{sobolevoso} and using the convergences at our disposal from Proposition \ref{lm_convergence_ueps}, we obtain
\[
\int_{B_{R/2}} \left|\nabla \left(|U_{x_i}|^\frac{p-2}{2}\,U_{x_i}\right)\right|^2\,dx\le \frac{C}{R^2}\,\int_{B_R} |\nabla U|^p\,dx,
\]
which is \eqref{apriorialpha} for $\alpha=p/2$.
In order to prove \eqref{apriorialpha} for a general $\alpha>p/2$, it is sufficient to observe that 
\begin{equation}
\label{bellepoque}
|U_{x_i}|^{\alpha-1}\,U_{x_i}=\left||U_{x_i}|^\frac{p-2}{2}\,U_{x_i}\right|^{\frac{2}{p}\,\alpha-1}\,|U_{x_i}|^\frac{p-2}{2}\,U_{x_i},
\end{equation}
and the function $t\mapsto |t|^{(2\,\alpha-p)/p}\,t$ is $C^1$. By using that 
\[
|U_{x_i}|^\frac{p-2}{2}\,U_{x_i}\in W^{1,2}_{\rm loc}(\Omega)\cap L^\infty_{\rm loc}(\Omega),
\] 
we get that $|U_{x_i}|^{\alpha-1}\,U_{x_i}\in W^{1,2}_{\rm loc}(\Omega)\cap L^\infty_{\rm loc}(\Omega)$ as well. Finally, to prove the estimate, we observe that  \eqref{bellepoque} implies
\[
\int_{B_{R/2}} \left|\nabla \left(|U_{x_i}|^{\alpha-1}\,U_{x_i}\right)\right|^2\,dx\le C\,\|U_{x_i}\|^{2\,\alpha-p}_{L^\infty(B_{R/2})} \int_{B_{R/2}} \left|\nabla \left(|U_{x_i}|^\frac{p-2}{2}\,U_{x_i}\right)\right|^2\,dx.
\]
By using \eqref{apriori} and \eqref{apriorialpha} for $\alpha=p/2$, we get the desired conclusion.
\end{proof}

We proceed with a technical result which will be needed to handle the case $p>2$.
\begin{lm}
\label{lm:schwarz2}
Let $p>2$ and let $U\in W^{1,p}_{\rm loc}(\Omega)$ still denote a local minimizer of $\mathfrak{F}$. Let $\beta\in\mathbb{R}$ and set
\[
F(t)=\frac{p}{2}\,\int_\beta^t |s|^\frac{p-2}{2}\,(s-\beta)_+\,ds,\qquad t\in\mathbb{R}.
\]
Then $F(U_{x_j})\in W^{1,2}_{\rm loc}(\Omega)$ and we have
\begin{equation}
\label{dellafrancesca}
\left(|U_{x_j}|^\frac{p-2}{2}\,U_{x_j}\right)_{x_k}\,(U_{x_j}-\beta)_+=\left(F(U_{x_j})\right)_{x_k}, \quad \mbox{ almost everywhere in }\Omega.
\end{equation}
\end{lm}
\begin{proof}
In order to prove that $F(U_{x_j})\in W^{1,2}_{\rm loc}(\Omega)$, we can observe that if we introduce the function
\[
G(t)=F\left(|t|^\frac{2-p}{p}\,t\right)=\frac{p}{2}\,\int_\beta^{|t|^\frac{2-p}{p}\,t} |s|^\frac{p-2}{2}\,(s-\beta)_+\,ds,
\]
then we have
\begin{equation}
\label{composition}
F(U_{x_j})=G\left(|U_{x_j}|^\frac{p-2}{2}\,U_{x_j}\right).
\end{equation}
With the simple change of variable $\tau=|s|^{(p-2)/2}\,s$, the function $G$ can be rewritten as
\[
G(t)=\int_{|\beta|^\frac{p-2}{2}\,\beta}^{t} \,\left(|\tau|^\frac{2-p}{p}\,\tau-\beta\right)_+\,d\tau.
\]
Hence, $G$ is a $C^1$  function. By using Theorem \ref{teo:apriori} and \eqref{composition}, we thus get that $F(U_{x_j})\in W^{1,2}_{\rm loc}(\Omega)$.
\par 
In order to prove \eqref{dellafrancesca}, we use the approximation scheme introduced in this section.
For every \(\varepsilon >0\), thanks to the smoothness of $u^\varepsilon$, we have
\begin{equation}
\label{eq1430}
\left(|u^\varepsilon_{x_j}|^\frac{p-2}{2}\,u^\varepsilon_{x_j}\right)_{x_k}\,(u^{\varepsilon}_{x_j}-\beta)_+ = \left(F(u_{x_j}^{\varepsilon})\right)_{x_k}.
\end{equation}
By Proposition \ref{lm_convergence_ueps}, we know that \(\nabla u^{\varepsilon}\) converges to \(\nabla U\) strongly in $L^p(B_r)$ and
\[
|u^{\varepsilon}_{x_j}|^{\frac{p-2}{2}}\,u^{\varepsilon}_{x_j}\ \mbox{ weakly converges in $W^{1,2}(B_r)$ to }\ |U_{x_j}|^{\frac{p-2}{2}}\,U_{x_j}.
\]
This implies that the left-hand side of \eqref{eq1430} converges weakly in $L^1(B_r)$ to the left-hand side of \eqref{dellafrancesca}. 
\par 
By using the uniform bounds of Lemma \ref{lm:regularity}, the local Lipschitz character of $G$ and the relation \eqref{composition}, we get
\[
\int_{B_r} |\nabla F(u^\varepsilon_{x_j})|^2\,dx=\int_{B_r} \left|\nabla G\left(|u^\varepsilon_{x_j}|^\frac{p-2}{2}\,u^\varepsilon_{x_j}\right)\right|^2\,dx\le C\, \int_{B_r} \left|\nabla \left(|u^\varepsilon_{x_j}|^\frac{p-2}{2}\,u^\varepsilon_{x_j}\right)\right|^2\,dx\le C,
\]
and
\[
\begin{split}
\lim_{\varepsilon\to 0}\int_{B_r} \left|F(U_{x_j})-F(u^\varepsilon_{x_j})\right|^2\,dx&=\lim_{\varepsilon\to 0}\int_{B_r} \left|G\left(|U_{x_j}|^\frac{p-2}{2}\,U_{x_j}\right)-G\left(|u^\varepsilon_{x_j}|^\frac{p-2}{2}\,u^\varepsilon_{x_j}\right)\right|^2\,dx\\ 
&\le C\, \lim_{\varepsilon\to 0}\int_{B_r} \left||U_{x_j}|^\frac{p-2}{2}\,U_{x_j}-|u^\varepsilon_{x_j}|^\frac{p-2}{2}\,u^\varepsilon_{x_j}\right|^2\,dx=0,
\end{split}
\]
where we used Proposition \ref{lm_convergence_ueps} for the last limit. We thus obtain that $F(u^\varepsilon_{x_j})$ converges weakly in $W^{1,2}(B_r)$ and strongly in $L^2(B_r)$ to $F(U_{x_j})$. We can then pass to the limit in the right-hand side of \eqref{eq1430}. 
\end{proof}

We end this subsection with two results on the solutions \(u^\varepsilon\) of the problem \eqref{approximation}. The first one is a standard minimum principle.

\begin{lm}[A minimum principle, $p>2$]
\label{lm:minimum}
With the notation above, let $B_r\Subset B$.
We have
\[
|u_{x_j}^\varepsilon|^\frac{p-2}{2}\,u_{x_j}^\varepsilon\ge C,\ \mbox{ on }\partial B_r\qquad \Longrightarrow\qquad |u_{x_j}^\varepsilon|^\frac{p-2}{2}\,u_{x_j}^\varepsilon\ge C,\ \mbox{ in } B_r.
\]
\end{lm}
\begin{proof}
In the differentiated equation \eqref{Differentiation_Euler_equation_varepsilon} we insert the test function
\[
\Phi=\left\{\begin{array}{cl}
\left(C-|u^\varepsilon_{x_j}|^\frac{p-2}{2}\,u^\varepsilon_{x_j}\right)_+&, \mbox{ in } B_r,\\
0&, \mbox{ in } B\setminus B_r,
\end{array}
\right.
\]
which is admissible thanks to the hypothesis. Observe that 
\begin{equation}
\label{1}
|u^\varepsilon_{x_j}|^\frac{p-2}{2}\,u^\varepsilon_{x_j}\le C\quad \Longleftrightarrow \quad u^\varepsilon_{x_j}\le |C|^\frac{2-p}{p}\,C,
\end{equation}
thus we obtain
\[
\sum_{i=1}^2\int_{\left\{u^\varepsilon_{x_j}\le |C|^\frac{2-p}{p}\,C\right\}\cap B_r}\Big(|u^\varepsilon_{x_i}|^{p-2}+\varepsilon\Big)\,|u^\varepsilon_{x_j}|^\frac{p-2}{2}\,\left|u^\varepsilon_{x_i x_j}\right|^2\,dx=0.
\]
Observe that the two terms are non-negative, thus for $i=j$ we can also infer
\[
\begin{split}
0&=\int_{\left\{u^\varepsilon_{x_j}\le |C|^\frac{2-p}{p}\,C\right\}\cap B_r}|u^\varepsilon_{x_j}|^{\frac{3}{2}\,(p-2)}\,\left|u^\varepsilon_{x_j x_j}\right|^2\,dx\\
&=\left(\frac{4}{3\,p-2}\right)^2\,\int_{\left\{u^\varepsilon_{x_j}\le |C|^\frac{2-p}{p}\,C\right\}\cap B_r} \,\left|\left(|u^\varepsilon_{x_j}|^{\frac{3}{4}\,(p-2)}\,u^\varepsilon_{x_j}\right)_{x_j}\right|^2\,dx\\
&=\left(\frac{4}{3\,p-2}\right)^2\,\int_{B_r} \,\left|\left(\min\left\{|u^\varepsilon_{x_j}|^{\frac{3}{4}\,(p-2)}\,u^\varepsilon_{x_j},\,|C|^\frac{p-2}{2\,p}\,C\right\}\right)_{x_j}\right|^2\,dx,
\end{split}
\]
where we used that
\begin{equation}
\label{2}
u^\varepsilon_{x_j}\le |C|^\frac{2-p}{p}\,C\qquad\Longleftrightarrow \qquad |u^\varepsilon_{x_j}|^{\frac{3}{4}\,(p-2)}\,u^\varepsilon_{x_j}\le |C|^\frac{p-2}{2\,p}\,C.
\end{equation}
This entails that
\[
\left(\min\left\{|u^\varepsilon_{x_j}|^{\frac{3}{4}\,(p-2)}\,u^\varepsilon_{x_j},\,|C|^\frac{p-2}{2\,p}\,C\right\}\right)_{x_j}=0,\qquad \mbox{ a.\,e. in } B_r,
\]
so that the Sobolev function
\[
\min\left\{|u^\varepsilon_{x_j}|^{\frac{3}{4}\,(p-2)}\,u^\varepsilon_{x_j},\,|C|^\frac{p-2}{2\,p}\,C\right\},
\]
does not depend on the variable $x_j$ in $B_r$. By assumption, this function is constant on $\partial B_r$. The last two facts imply that
\[
\min\left\{|u^\varepsilon_{x_j}|^{\frac{3}{4}\,(p-2)}\,u^\varepsilon_{x_j},\,|C|^\frac{p-2}{2\,p}\,C\right\}=|C|^\frac{p-2}{2\,p}\,C,\qquad \mbox{ a.\,e. in } B_r,
\]
which is the desired conclusion, thanks to \eqref{1} and \eqref{2}.
\end{proof}
Finally, we will need the following result about convergence of traces.
\begin{lm}
\label{lm:tracce}
Let $B_r\Subset B$. With the notation above, there exists a sequence \(\{\varepsilon_k\}_{k\in \mathbb{N}}\subset (0,\varepsilon_0)\) such that for almost every $s\in[0,r]$, we have
\[
\lim_{k\to +\infty}\left\||u^{\varepsilon_k}_{x_j}|^\frac{p-2}{2}\,u^{\varepsilon_k}_{x_j}-|U_{x_j}|^\frac{p-2}{2}\,U_{x_j}\right\|_{L^\infty(\partial B_s)}=0,\qquad j=1,2.
\]
\end{lm}
\begin{proof}
We first observe that 
\[
\left\{|u^{\varepsilon}_{x_j}|^\frac{p-2}{2}\,u^{\varepsilon}_{x_j}-|U_{x_j}|^\frac{p-2}{2}\,U_{x_j}\right\}_{0<\varepsilon<\varepsilon_0},
\] 
weakly converges to \(0\) in $W^{1,2}(B_r)$, thanks to Proposition \ref{lm_convergence_ueps}. Thus for every $0<\tau<1$, there exists a subsequence which strongly converges to $0$ in the fractional Sobolev space $W^{\tau,2}(B_r)$. 
We take $1/2<\tau<1$ and observe that the previous convergence implies that we can extract again a subsequence which strongly converges to $0$ in $W^{\tau,2}(\partial B_s)$, for almost every $s\in[0,r]$ (see Lemma \ref{lm:dellafrancesca}). In order to conclude, it is now sufficient to use that for $1/2<\tau<1$, the space $W^{\tau,2}(\partial B_s)$ is continuously embedded in $C^0(\partial B_s)$ (since $\partial B_s$ is one-dimensional, see \cite[Theorem 7.57]{Ad}). 
\end{proof}

\subsection{Regularization scheme, case $1<p\le 2$}

In this case, the functional in \eqref{approximation} is not smooth enough, in particular is not $C^2$. Thus the regularized problem is now
\begin{equation}
\label{approximationsub}
\min\left\{\sum_{i=1}^2\frac{1}{p}\,\int_{B} \left(\varepsilon+|w_{x_i}|^2\right)^\frac{p}{2}\, :\, w-U_\varepsilon \in W^{1,p}_0(B)\right\}.
\end{equation}
This problem admits a unique solution \(u^\varepsilon\), which is smooth on \(\overline{B}\),  see again \cite[Theorem 2.4]{BBJ}. Moreover, the solution \(u^\varepsilon\) satisfies the corresponding Euler-Lagrange equation, i.\,e.
\begin{equation}
\label{nondiff}
\sum_{i=1}^2\int_{B} (\varepsilon+|u^\varepsilon_{x_i}|^2)^\frac{p-2}{2}\,u^\varepsilon_{x_i}\,\varphi_{x_i}  \,dx=0,\qquad \mbox{ for every }\varphi\in W^{1,p}_0(B).
\end{equation}
We still have the following uniform estimate. The proof is standard routine and is left to the reader.
\begin{lm}[Uniform energy estimate]
\label{lm:energybis}
There exists a constant $C=C(p)>0$ such that for every $0<\varepsilon<\varepsilon_0$ the following estimate holds 
\begin{equation}
\label{uniformegbis}
\int_B |\nabla u^\varepsilon|^p\, dx\le C\left(\int_{\lambda_B\,B} |\nabla U|^p\,dx+\varepsilon^{\frac{p}{2}}\,|B|\right).
\end{equation}
Moreover, the family $\{u^{\varepsilon}\}_{0<\varepsilon<\varepsilon_0}$ converges weakly in $W^{1,p}(B)$ and strongly in $L^p(B)$ to $U$.
\end{lm}

We will rely on the following Cacciopoli inequality to obtain certain bounds on the family \(\{u^\varepsilon\}_{0<\varepsilon<\varepsilon_0}\).
\begin{prop}[Caccioppoli inequality for the gradient, $1<p\le 2$]
\label{lm:caccioespilon2}
Let $\zeta:\mathbb{R}\to \mathbb{R}$ be a $C^1$ monotone function, then for every $\eta\in C^2$ with compact support in $B$ we have
\begin{equation}
\label{uniformesobsub2}
\begin{split}
\sum_{i=1}^2\int & (\varepsilon+|u^\varepsilon_{x_i}|^2)^\frac{p-2}{2}\,|\zeta'(u^\varepsilon_{x_j})|\,\left|u^\varepsilon_{x_j\,x_i}\right|^2\,\eta^2\,dx\\
& \le C\, \int (\varepsilon+|\nabla u^\varepsilon|^2)^\frac{p}{2}\,|\zeta'(u^\varepsilon_{x_j})|\,|\nabla \eta|^2\,dx\\
&+C\,\int (\varepsilon+|\nabla u^\varepsilon|^2)^\frac{p-1}{2}\,|\zeta(u^\varepsilon_{x_j})|\,\Big(|\nabla \eta|^2+|D^2\eta|\Big)\,dx,
\end{split}
\end{equation}
for some $C=C(p)>0$.
\end{prop}

\begin{proof}
We suppose that $\zeta\in C^2$, then the general result can be obtained with a standard approximation argument.
In order to obtain \eqref{uniformesobsub2}, we use a trick by Fonseca and Fusco \cite{FF} in order to avoid the use of the upper bound on the Hessian of 
\[
H_\varepsilon(t):=\frac{1}{p}\,(\varepsilon+t^2)^\frac{p}{2},\qquad t\in\mathbb{R},
\]
see also \cite{EM} and \cite{FFM}. 
\par
We start by testing \eqref{nondiff} against $\varphi=(\zeta(u^\varepsilon_{x_j})\,\eta^2)_{x_j}$. Thus we get 
\[
\sum_{i=1}^2\int_{B}  (\varepsilon+|u^\varepsilon_{x_i}|^2)^\frac{p-2}{2}\,u^\varepsilon_{x_i}\,(\zeta(u^\varepsilon_{x_j})\,\eta^2)_{x_j\,x_i}  \,dx=0.
\]
By using the smoothness of $u^\varepsilon$ and $\eta$, we have
\[
\begin{split}
(\zeta(u^\varepsilon_{x_j})\,\eta^2)_{x_j\,x_i}&=(\zeta(u^\varepsilon_{x_j})\,\eta^2)_{x_i\,x_j}\\
&=\Big(\zeta'(u^\varepsilon_{x_j})\,u^\varepsilon_{x_j\,x_i}\,\eta^2+2\,\zeta(u^\varepsilon_{x_j})\,\eta\,\eta_{x_i}\Big)_{x_j}\\
&=\Big(\zeta'(u^\varepsilon_{x_j})\,u^\varepsilon_{x_j\,x_i}\,\eta^2\Big)_{x_j}+2\,\Big(\zeta(u^\varepsilon_{x_j})\,\eta\,\eta_{x_i}\Big)_{x_j}.
\end{split}
\]
By using an integration by parts, we thus obtain
\[
\begin{split}
-\sum_{i=1}^2\int_{B}& \Big( (\varepsilon+|u^\varepsilon_{x_i}|^2)^\frac{p-2}{2}\,u^\varepsilon_{x_i}\Big)_{x_j}\,\zeta'(u^\varepsilon_{x_j})\,u^\varepsilon_{x_j\,x_i}\,\eta^2\,dx\\
&+2\,\sum_{i=1}^2\int_{B}  (\varepsilon+|u^\varepsilon_{x_i}|^2)^\frac{p-2}{2}\,u^\varepsilon_{x_i}\,(\zeta(u^\varepsilon_{x_j})\,\eta\,\eta_{x_i})_{x_j}  \,dx=0.
\end{split}
\]
With simple manipulations, this becomes
\begin{equation}
\label{reaccrocher}
\begin{split}
\sum_{i=1}^2\int_{B}& \Big(\varepsilon+|u^\varepsilon_{x_i}|^2\Big)^\frac{p-2}{2}\,\zeta'(u^\varepsilon_{x_j})\,\left|u^\varepsilon_{x_j\,x_i}\right|^2\,\eta^2\,dx\\
&+(p-2)\,\sum_{i=1}^2\int_{B}\Big(\varepsilon+|u^\varepsilon_{x_i}|^2\Big)^\frac{p-4}{2}\,|u^\varepsilon_{x_i}|^2\,\zeta'(u^\varepsilon_{x_j})\,\left|u^\varepsilon_{x_j\,x_i}\right|^2\,\eta^2\,dx\\
&=2\,\sum_{i=1}^2\int_{B} (\varepsilon+|u^\varepsilon_{x_i}|^2)^\frac{p-2}{2}\,u^\varepsilon_{x_i}\,\zeta'(u^\varepsilon_{x_j})\,u^\varepsilon_{x_j\,x_j}\,\eta\,\eta_{x_i} \,dx\\
&+2\,\sum_{i=1}^2\int_{B} (\varepsilon+|u^\varepsilon_{x_i}|^2)^\frac{p-2}{2}\,u^\varepsilon_{x_i}\,\zeta(u^\varepsilon_{x_j})\,(\eta\,\eta_{x_i})_{x_j}  \,dx.
\end{split}
\end{equation}
We now observe that 
\[
\begin{split}
\sum_{i=1}^2\int_{B} &(\varepsilon+|u^\varepsilon_{x_i}|^2)^\frac{p-2}{2}\,\zeta'(u^\varepsilon_{x_j})\,\left|u^\varepsilon_{x_j\,x_i}\right|^2\,\eta^2\,dx\\
&+(p-2)\,\sum_{i=1}^2\int_{B}\Big(\varepsilon+|u^\varepsilon_{x_i}|^2\Big)^\frac{p-4}{2}\,|u^\varepsilon_{x_i}|^2\,\zeta'(u^\varepsilon_{x_j})\,\left|u^\varepsilon_{x_j\,x_i}\right|^2\,\eta^2\,dx\\
&=\sum_{i=1}^2\int_{B} (\varepsilon+|u^\varepsilon_{x_i}|^2)^\frac{p-4}{2}\,(\varepsilon+(p-1)\,|u_{x_i}^\varepsilon|^2)\,\zeta'(u^\varepsilon_{x_j})\,\left|u^\varepsilon_{x_j\,x_i}\right|^2\,\eta^2\,dx
\end{split}
\]
so that the left-hand side of \eqref{reaccrocher} has a sign. Thus we obtain\footnote{Recall that by hypothesis, $\zeta'$ has constant sign.}
\begin{equation}
\label{reaccrocher2}
\begin{split}
\sum_{i=1}^2&\int_{B} (\varepsilon+|u^\varepsilon_{x_i}|^2)^\frac{p-4}{2}\,(\varepsilon+(p-1)\,|u_{x_i}^\varepsilon|^2)\,|\zeta'(u^\varepsilon_{x_j})|\,\left|u^\varepsilon_{x_j\,x_i}\right|^2\,\eta^2\,dx\\
&\le 2\,\sum_{i=1}^2\int_{B} (\varepsilon+|u^\varepsilon_{x_i}|^2)^\frac{p-2}{2}\,|u^\varepsilon_{x_i}|\,|\zeta'(u^\varepsilon_{x_j})|\,\left|u^\varepsilon_{x_j\,x_j}\right|\,\eta\,|\eta_{x_i}| \,dx\\
&+2\,\sum_{i=1}^2\int_{B} (\varepsilon+|u^\varepsilon_{x_i}|^2)^\frac{p-2}{2}\,|u^\varepsilon_{x_i}|\,|\zeta(u^\varepsilon_{x_j})|\,\left|(\eta\,\eta_{x_i})_{x_j}\right|  \,dx.
\end{split}
\end{equation}
We now estimate the left-hand side of \eqref{reaccrocher2} from below
\[
\begin{split}
\sum_{i=1}^2&\int_{B} (\varepsilon+|u^\varepsilon_{x_i}|^2)^\frac{p-4}{2}\,(\varepsilon+(p-1)\,|u_{x_i}^\varepsilon|^2)\,|\zeta'(u^\varepsilon_{x_j})|\,\left|u^\varepsilon_{x_j\,x_i}\right|^2\,\eta^2\,dx\\
&\ge (p-1)\,\sum_{i=1}^2\int_{B} (\varepsilon+|u^\varepsilon_{x_i}|^2)^\frac{p-2}{2}\,|\zeta'(u^\varepsilon_{x_j})|\,\left|u^\varepsilon_{x_j\,x_i}\right|^2\,\eta^2\,dx\\
&\ge \frac{p-1}{2}\,\sum_{i=1}^2\int_{B} (\varepsilon+|u^\varepsilon_{x_i}|^2)^\frac{p-2}{2}\,|\zeta'(u^\varepsilon_{x_j})|\,\left|u^\varepsilon_{x_j\,x_i}\right|^2\,\eta^2\,dx\\
&+\frac{p-1}{2}\,\sum_{i=1}^2\int_{B} (\varepsilon+|\nabla u^\varepsilon|^2)^\frac{p-2}{2}\,|\zeta'(u^\varepsilon_{x_j})|\,\left|u^\varepsilon_{x_j\,x_i}\right|^2\,\eta^2\,dx,
\end{split}
\]
where we used that $p-2<0$. We will use the last term as a {\it sponge term} in order to absorb the second derivatives of $u^\varepsilon$ contained in the right-hand side.
\par
As for the first term in the right-hand side of \eqref{reaccrocher2}
\[
\begin{split}
\int_{B} (\varepsilon+|u^\varepsilon_{x_i}|^2)^\frac{p-2}{2}&\,|u^\varepsilon_{x_i}|\,|\zeta'(u^\varepsilon_{x_j})|\,\left|u^\varepsilon_{x_j\,x_j}\right|\,\eta\,|\eta_{x_i}| \,dx\\
&\le \int_{B} (\varepsilon+|u^\varepsilon_{x_i}|^2)^\frac{p-1}{2}\,|\zeta'(u^\varepsilon_{x_j})|\,|u^\varepsilon_{x_j\,x_j}|\,\eta\,|\eta_{x_i}| \,dx\\
&\le \int_{B} (\varepsilon+|\nabla u^\varepsilon|^2)^\frac{p-1}{2}\,|\zeta'(u^\varepsilon_{x_j})|\,|u^\varepsilon_{x_j\,x_j}|\,\eta\,|\eta_{x_i}| \,dx\\
&\le \frac{1}{2\,\tau}\,\int_{B_R} (\varepsilon+|\nabla u^\varepsilon|^2)^\frac{p}{2}\,|\zeta'(u^\varepsilon_{x_j})|\,|\nabla \eta|^2\,dx\\
&+\frac{\tau}{2} \int_{B}(\varepsilon+|\nabla u^\varepsilon|^2)^\frac{p-2}{2}\, |\zeta'(u^\varepsilon_{x_j})|\,\left|u^\varepsilon_{x_j\,x_j}\right|^2\,\eta^2\,dx.
\end{split}
\]
Also, for the last term of \eqref{reaccrocher2}, we simply get 
\[
\begin{split}
\int_{B} &(\varepsilon+|u^\varepsilon_{x_i}|^2)^\frac{p-2}{2}\,|u^\varepsilon_{x_i}|\,|\zeta(u^\varepsilon_{x_j})|\,\left|(\eta\,\eta_{x_i})_{x_j}\right|\,dx\\
&\le \int_{B_R} (\varepsilon+|\nabla u^\varepsilon|^2)^\frac{p-1}{2}\,|\zeta(u^\varepsilon_{x_j})|\,\Big(|\nabla \eta|^2+|D^2\eta|\Big)\,dx.\\
\end{split}
\]
By using these estimates in \eqref{reaccrocher} and taking $\tau=(p-1)/2$ in order to absorb the Hessian term on the right-hand side, we obtain
\begin{equation}
\label{stimarella2}
\begin{split}
\sum_{i=1}^2\int_{B}& (\varepsilon+|u^\varepsilon_{x_i}|^2)^\frac{p-2}{2}\,|\zeta'(u^\varepsilon_{x_j})|\,\left|u^\varepsilon_{x_j\,x_i}\right|^2\,\eta^2\,dx\\
& \le C\,\int_{B_R} (\varepsilon+|\nabla u^\varepsilon|^2)^\frac{p}{2}\,|\zeta'(u^\varepsilon_{x_j})|\,|\nabla \eta|^2\,dx\\
&+C\,\int_{B_R} (\varepsilon+|\nabla u^\varepsilon|^2)^\frac{p-1}{2}\,|\zeta(u^\varepsilon_{x_j})|\,\Big(|\nabla \eta|^2+|D^2\eta|\Big)\,dx,
\end{split}
\end{equation}
which is exactly \eqref{uniformesobsub2}.
\end{proof}

We now collect some bounds on the family \(\{u_{\varepsilon}\}_{0<\varepsilon<\varepsilon_0}\).
\begin{lm}[Uniform estimates, $1<p\le 2$]
\label{lm:lemma1p2}
Let $1<p\le 2$, then for every $B_r\Subset B$ we have
\begin{equation}
\label{uniformelipsub}
\|u^\varepsilon\|_{L^\infty(B_r)}+\|\nabla u^\varepsilon\|_{L^\infty(B_r)}\le C,
\end{equation}
\begin{equation}
\label{uniformesobsub}
\sum_{i=1}^2\int_{B_r} \left(\varepsilon+|u^\varepsilon_{x_i}|^2\right)^\frac{p-2}{2}\,\left|u^\varepsilon_{x_i\,x_j}\right|^2\le C,\qquad j=1,2,
\end{equation}
and
\begin{equation}
\label{uniformeH2sub}
\int_{B_r} |\nabla u^\varepsilon_{x_j}|^2\,dx\le C,\qquad j=1,2,
\end{equation}
for some $C>0$ independent of $\varepsilon$.
\end{lm}
\begin{proof}
The $L^\infty$ estimate can  be found in \cite[Chapter 7]{Gi} again,  while the Lipschitz estimate follows from \cite[Theorem 2.2]{FF}. More precisely,  for every ball \(B_s\) such that \(B_{2s}\Subset B\), 
\begin{equation}\label{eq_FF_eps}
\sup_{B_{s}}\left(\varepsilon + |\nabla u^{\varepsilon}|^2\right)^{\frac{p}{2}}\,dx \leq C\, \fint_{B_{2s}}\left(\varepsilon + |\nabla u^\varepsilon|\right)^{\frac{p}{2}}\,dx.
\end{equation}
By covering a given ball \(B_r\Subset B\) by a finite number of balls \(B_s\) such that \(B_{2s}\Subset B\) and using the bound on the $L^p$ norm of $\nabla u^\varepsilon$, one easily gets the Lipschitz estimate in \eqref{uniformelipsub} for some constant \(C>0\) which may depend on \(B_r\) but not on \(\varepsilon\).
\vskip.2cm\noindent
In order to prove \eqref{uniformesobsub}, we introduce two balls \(B_r\Subset B_R\Subset B\) and a standard cut-off function $\eta\in C^2$ such that
\[
0\le \eta\le 1,\qquad\eta  \equiv 1 \textrm{ on } B_r, \qquad \eta \equiv 0 \textrm{ on } \mathbb{R}^2\setminus B_{R},
\]
\[
|\nabla \eta|\le\frac{C}{R-r},\qquad |D^2 \eta| \leq \frac{C}{(R-r)^2}.
\]
By taking $\zeta(t)=t$ in \eqref{uniformesobsub2}, one gets
\begin{equation}
\label{stimarella}
\begin{split}
\sum_{i=1}^2\int & (\varepsilon+|u^\varepsilon_{x_i}|^2)^\frac{p-2}{2}\,\left|u^\varepsilon_{x_j\,x_i}\right|^2\,\eta^2\,dx\\
& \le C\int (\varepsilon+|\nabla u^\varepsilon|^2)^\frac{p}{2}\,|\nabla \eta|^2\,dx+C\int (\varepsilon+|\nabla u^\varepsilon|^2)^\frac{p-1}{2}\,|u^\varepsilon_{x_j}|\,\Big(|\nabla \eta|^2+|D^2\eta|\Big)\,dx.
\end{split}
\end{equation}
By recalling the uniform bound on the $L^p$ norm of $\nabla u^\varepsilon$, \eqref{stimarella} gives \eqref{uniformesobsub}.
\vskip.2cm\noindent
We now observe that 
\[
\begin{split}
\sum_{i=1}^2\int_{B} (\varepsilon+|u^\varepsilon_{x_i}|^2)^\frac{p-2}{2}\,\left|u^\varepsilon_{x_j\,x_i}\right|^2\,\eta^2\,dx&\ge\int_{B} (\varepsilon+|\nabla u^\varepsilon|^2)^\frac{p-2}{2}\,\left|u^\varepsilon_{x_j\,x_i}\right|^2\,\eta^2\,dx\\
&\ge \left(\varepsilon+\|\nabla u^\varepsilon\|^2_{L^\infty(B_R)}\right)^\frac{p-2}{2}\int_{B_r}\left|u^\varepsilon_{x_j\,x_i}\right|^2\,dx.
\end{split}
\]
By appealing to \eqref{stimarella}, this yields
\[
\int_{B_r}\left|u^\varepsilon_{x_j\,x_i}\right|^2\,dx\le \frac{C}{(R-r)^2}\,\left(\varepsilon+\|\nabla u^\varepsilon\|^2_{L^\infty(B_R)}\right)^\frac{2-p}{2}\,\int_{B_R} (\varepsilon+|\nabla u^\varepsilon|^2)^\frac{p}{2}\,dx.
\]
In order to conclude, it is sufficient to use \eqref{uniformelipsub} for the ball $B_R\Subset B$ and again the uniform estimate on the $L^p$ norm of $\nabla u^\varepsilon$.
\end{proof}

\begin{prop}
\label{prop:convergence_uep1p2bis}
With the notation above, 
for every $B_r\Subset B$, we have:
\vskip.2cm
\begin{enumerate}
\item \(\{u^{\varepsilon}\}_{0<\varepsilon<\varepsilon_0}\) converges uniformly to \(U\) on \(\overline{B_r}\);
\vskip.2cm
\item \(\{\nabla u^{\varepsilon}\}_{0<\varepsilon<\varepsilon_0}\) converges to \(\nabla U\) weakly in $W^{1,2}(B_r)$ and strongly in $L^2(B_r)$. In particular, we have
\[
U_{x_i}\in W^{1,2}(B_r);
\]
\vskip.2cm
\item $\left\{(\varepsilon+|u^{\varepsilon}_{x_i}|^2)^{\frac{p-2}{4}}\,u^{\varepsilon}_{x_i}\right\}_{0<\varepsilon<\varepsilon_0}$ converges to \(|U_{x_i}|^{\frac{p-2}{2}}\,U_{x_i}\) weakly in $W^{1,2}(B_r)$ and strongly in $L^{4/p}(B_r)$. In particular, we have
\[
|U_{x_i}|^{\frac{p-2}{2}}\,U_{x_i}\in W^{1,2}(B_r). 
\]
\end{enumerate}
\end{prop}
\begin{proof}
We already know from Lemma \ref{lm:energybis} that $u^\varepsilon$ converges to $U$ weakly in $W^{1,p}(B)$ and strongly in $L^p(B)$.
\par
By \eqref{uniformelipsub} and the Arzel\`a-Ascoli Theorem,  the convergence of $\{u^{\varepsilon}\}_{0<\varepsilon<\varepsilon_0}$ to $U$ is uniform on $\overline B_r$, for every $B_r\Subset B$.
\par
From estimates \eqref{uniformelipsub} and \eqref{uniformeH2sub}, we get that $\{ u^{\varepsilon}_{x_i}\}_{0<\varepsilon<\varepsilon_0}$ is uniformly bounded in $W^{1,2}(B_r)$. By Rellich-Kondra\v{s}ov Theorem, we can infer strong convergence in $L^2(B_r)$ to $U_{x_i}$, for every $i=1,2$.
\par
We now observe that
\[
\begin{split}
\left|\nabla \left((\varepsilon+|u^{\varepsilon}_{x_i}|^2)^{\frac{p-2}{4}}\,u^{\varepsilon}_{x_i}\right)\right|^2&\le 2\,\left|\frac{p-2}{2}\,(\varepsilon+|u^{\varepsilon}_{x_i}|^2)^{\frac{p-6}{4}}\,\left|u^{\varepsilon}_{x_i}\right|^2\,\nabla u^{\varepsilon}_{x_i}\right|^2\\
&+2\,\left|(\varepsilon+|u^{\varepsilon}_{x_i}|^2)^{\frac{p-2}{4}}\,\nabla u^{\varepsilon}_{x_i}\right|^2\le C\,\sum_{j=1}^2 (\varepsilon+|u^{\varepsilon}_{x_i}|^2)^{\frac{p-2}{2}}\,| u^{\varepsilon}_{x_i\,x_j}|^2.
\end{split}
\]
By \eqref{uniformesobsub}, this implies that
\begin{equation}\label{eq_sequence_weird}
\left\{(\varepsilon+|u^{\varepsilon}_{x_i}|^2)^{\frac{p-2}{4}}\,u^{\varepsilon}_{x_i}\right\}_{0<\varepsilon<\varepsilon_0},\qquad i=1,2,
\end{equation}
is bounded in $W^{1,2}(B_r)$. Again by Rellich-Kondra\v{s}ov Theorem we can assume that, up to a subsequence (we do not relabel), it converges to some function \(V_i\in W^{1,2}(B_r)\), weakly in \(W^{1,2}(B_r)\) and strongly in \(L^2(B_r)\). We now show at the same time that \(V_i=|U_{x_i}|^{\frac{p-2}{2}}U_{x_i}\) and that actually we have strong convergence in $L^{4/p}(B_r)$. Indeed, by using the elementary inequality of Corollary \ref{coro:dibene},
we obtain
\[
\begin{split}
\int_{B_r} &\left|(\varepsilon+|u^{\varepsilon}_{x_i}|^2)^{\frac{p-2}{4}}\,u^{\varepsilon}_{x_i}-|U_{x_i}|^{\frac{p-2}{2}}\,U_{x_i}\right|^\frac{4}{p}\,dx\\
&\le C\, \int_{B_r} \left|(\varepsilon+|u^{\varepsilon}_{x_i}|^2)^{\frac{p-2}{4}}\,u^{\varepsilon}_{x_i}-(\varepsilon+|U_{x_i}|^2)^{\frac{p-2}{4}}\,U_{x_i}\right|^\frac{4}{p}\,dx\\
&+ C\,\int_{B_r} \left|(\varepsilon+|U_{x_i}|^2)^{\frac{p-2}{4}}\,U_{x_i}-|U_{x_i}|^\frac{p-2}{2}\,U_{x_i}\right|^\frac{4}{p}\,dx\\
&\le C\, \int_{B_r} \left|u^{\varepsilon}_{x_i}-U_{x_i}\right|^2\,dx+ C\,\int_{B_r} \left|(\varepsilon+|U_{x_i}|^2)^{\frac{p-2}{4}}\,U_{x_i}-|U_{x_i}|^\frac{p-2}{2}\,U_{x_i}\right|^\frac{4}{p}\,dx.
\end{split}
\]
By using the strong convergence of the gradients proved above (for the first term) and the Dominated Convergence Theorem (for the second one), this implies that \(V_i=|U_{x_i}|^{\frac{p-2}{2}}U_{x_i}\) and the convergence of the full original sequence in \eqref{eq_sequence_weird}, weakly in \(W^{1,2}(B_r)\) and strongly in \(L^{4/p}(B_r)\). The proof is complete.
\end{proof}

Using the above convergence result, one can establish  the following regularity properties for the solution \(U\). 

\begin{teo}[A priori estimates, $1<p\le 2$]
\label{teo:apriorisub}
Every local minimizer $U\in W^{1,p}_{\rm loc}(\Omega)$ of the functional $\mathfrak{F}$ is a locally Lipschitz function, such that for every $\alpha\ge p/2$ we have
\[
|U_{x_i}|^{\alpha-1}\,U_{x_i}\in W^{1,2}_{\rm loc}(\Omega),\qquad i=1,2.
\]  
In particular, we have $\nabla U\in W^{1,2}_{\rm loc}(\Omega;\mathbb{R}^2)$.
Moreover, for every $B_R\Subset\Omega$, we have
\begin{equation}
\label{apriorisub}
\|U_{x_j}\|_{L^\infty(B_{R/2})}\le C_1\,\left(\fint_{B_R} |\nabla U|^p\,dx\right)^\frac{1}{p},\qquad j=1,2,
\end{equation}
\begin{equation}
\label{apriorisub2}
\int_{B_{R/2}} \left|\nabla \left(|U_{x_j}|^{\alpha-1}\,U_{x_j}\right)\right|^2\,dx\le C_2\,\left(\fint_{B_R} |\nabla U|^p\,dx\right)^\frac{2\,\alpha}{p},\qquad j=1,2,
\end{equation}
for some $C_1=C_1(p)>0$ and $C_2=C_2(p,\alpha)>0$.
\end{teo}
\begin{proof}
Local Lipschitz regularity and the scaling invariant estimate \eqref{apriorisub} follow from \cite[Theorem 2.2]{FF}.
\vskip.2cm\noindent
We already know from Proposition \ref{prop:convergence_uep1p2bis} that $|U_{x_i}|^{(p-2)/2}\,U_{x_i} \in W^{1,2}_{\rm loc}(\Omega)$. 
In order to get \eqref{apriorisub2} for $\alpha=p/2$, we first observe that
\[
\left|\nabla \left(\left(\varepsilon+|u^{\varepsilon}_{x_j}|^2\right)^{\frac{p-2}{4}} u^{\varepsilon}_{x_j} \right)\right|^2\leq \left(\varepsilon+|u^{\varepsilon}_{x_j}|^2\right)^{\frac{p-2}{2}}|\nabla u^{\varepsilon}_{x_j}|^2.
\]
We multiply the above inequality with the cut-off function \(\eta^2\) as in \eqref{stimarella}, associated to the balls \(B_{R/2}\Subset B_{R}\). Integrating the resulting inequality, we get
\[
\int_{B_{R/2}} \left|\nabla \left(\left(\varepsilon+|u^{\varepsilon}_{x_j}|^2\right)^{\frac{p-2}{4}} u^{\varepsilon}_{x_j} \right) \right|^2\,dx
\leq 
\int_{B_{R}}\left(\varepsilon+|u^{\varepsilon}_{x_j}|^2\right)^{\frac{p-2}{2}}|\nabla u^{\varepsilon}_{x_j}|^2\eta^2\,dx.
\]
Using \eqref{stimarella}, this implies
\[
\int_{B_{R/2}} \left|\nabla \left(\left(\varepsilon+|u^{\varepsilon}_{x_j}|^2\right)^{\frac{p-2}{4}} u^{\varepsilon}_{x_j}\right) \right|^2\,dx
\leq
\frac{C}{R^2}\int_{B_{R}}\left(\varepsilon+|\nabla u^{\varepsilon}|^2\right)^{\frac{p}{2}}\,dx.
\]
By taking the limit in the previous inequality and using the convergences of Proposition \ref{prop:convergence_uep1p2bis}, we get \eqref{apriorisub2} for $\alpha=p/2$.
\par
The last part of the statement now follows as in Theorem \ref{teo:apriori} above (observe that this time $0<p/2\le 1$).
\end{proof}
\begin{oss}
For later reference, we observe that for every \(k,j=1,2\), 
\begin{equation}\label{eq_chain_rule_weird}
\left( |U_{x_j}|^{\frac{p-2}{2}}U_{x_j}\right)_{x_k} = \frac{p}{2}\,|U_{x_j}|^{\frac{p-2}{2}}U_{x_jx_k}
\qquad \textrm{a.\,e. on } \{U_{x_j}\not=0\}.
\end{equation}
Since the function \(t\mapsto |t|^{\frac{p-2}{2}}t\) is not \(C^1\) for \(1<p<2\), nor locally Lipschitz, the identity \eqref{eq_chain_rule_weird} does not follow from the chain rule in a straightforward way.
We start instead from the following identity, which results from the classical chain rule for smooth functions:
\begin{equation}\label{eq_dresden}
\left(\varepsilon +|u^{\varepsilon}_{x_j}|^2\right)^{\frac{2-p}{4}}\left((\varepsilon + |u^{\varepsilon}_{x_j}|^2)^{\frac{p-2}{4}}u_{x_j}^\varepsilon \right)_{x_k} =\left(\frac{\varepsilon+\dfrac{p}{2}\,|u^{\varepsilon}_{x_j}|^2}{\varepsilon + |u^{\varepsilon}_{x_j}|^2} \right)u^{\varepsilon}_{x_jx_k}.
\end{equation}
In the left-hand side, $(\varepsilon +|u^{\varepsilon}_{x_j}|^2)^{(2-p)/4}$ is uniformly bounded on \(B_R\Subset B\) and converges almost everywhere to \(|U_{x_j}|^{(2-p)/2}\), while
\[
\left((\varepsilon + |u^{\varepsilon}_{x_j}|^2)^{\frac{p-2}{4}}u_{x_j} \right)_{x_k}\ \mbox{ weakly converges in $L^{2}(B_R)$ to }\ \left(|U_{x_j}|^{\frac{p-2}{2}}U_{x_j}\right)_{x_k}.
\] 
Hence, the product converges weakly in \(L^2(B_R)\) to \(|U_{x_j}|^{(2-p)/2}\,( |U_{x_j}|^{(p-2)/2}\,U_{x_j})_{x_k}\). 
\par
A similar argument proves that the right-hand side of \eqref{eq_dresden} converges to \((p/2)\,U_{x_jx_k}\) weakly in \(L^2(B_R)\). We have thus proved that for almost every \(x\in B_R\),
\[
|U_{x_j}|^{\frac{2-p}{2}}\left( |U_{x_j}|^{\frac{p-2}{2}}U_{x_j}\right)_{x_k} = 
\frac{p}{2}\,U_{x_jx_k}.
\]
The identity \eqref{eq_chain_rule_weird} follows at once. 
\end{oss}
As in the case \(p> 2\), we end this subsection on the case \(1<p\le 2\) with two additional results on the solutions \(u^\varepsilon\) of the problem \ref{approximationsub}.
\begin{lm}[A minimum principle, $1<p\le 2$]
\label{lm:minimumbis}
 Let $B_r\Subset B$. With the notation above, we have
\[
u_{x_j}^\varepsilon\ge C,\ \mbox{ on }\partial B_r\qquad \Longrightarrow\qquad u_{x_j}^\varepsilon\ge C,\ \mbox{ in } B_r.
\]
\end{lm}
\begin{proof}
By inserting in \eqref{nondiff} a test function of the form $\varphi_{x_j}$ with $\varphi$ smooth with compact support in $B$ and integrating by parts, we get
\[
\begin{split}
\sum_{i=1}^2\int_{B}  \left((\varepsilon+|u^\varepsilon_{x_i}|^2)^\frac{p-2}{2}\,u^\varepsilon_{x_i}\right)_{x_j}\,\varphi_{x_i}  \,dx=0.
\end{split}
\]
This is the same as
\[
\begin{split}
\sum_{i=1}^2\int_{B}& (\varepsilon+|u^\varepsilon_{x_i}|^2)^\frac{p-2}{2}\,u^\varepsilon_{x_i\,x_j}\,\varphi_{x_i}  \,dx+(p-2)\,\sum_{i=1}^2\int_{B}  (\varepsilon+|u^\varepsilon_{x_i}|^2)^\frac{p-4}{2}\,\left|u^\varepsilon_{x_i}\right|^2\,u_{x_i\,x_j}^\varepsilon\,\varphi_{x_i}  \,dx=0.
\end{split}
\]
By regularity of $u^\varepsilon$, the previous identity is still true for functions $\varphi\in W^{1,2}_0(B)$.
In the previous identity, we insert the test function
\[
\Phi=\left\{\begin{array}{cl}
(C-u^\varepsilon_{x_j})_+&, \mbox{ in } B_r,\\
0&, \mbox{ in } B\setminus B_r,
\end{array}
\right.
\]
which is admissible thanks to the hypothesis on $u^\varepsilon_{x_j}$. We obtain
\[
\begin{split}
-\sum_{i=1}^2&\int_{\left\{x\in B_r\, :\, u_{x_j}^\varepsilon \le C\right\}} (\varepsilon+|u^\varepsilon_{x_i}|^2)^\frac{p-2}{2}\,\left|u^\varepsilon_{x_i\,x_j}\right|^2\,dx\\
&-(p-2)\,\sum_{i=1}^2\int_{\left\{x\in B_r\, :\, u_{x_j}^\varepsilon \le C\right\}}  (\varepsilon+|u^\varepsilon_{x_i}|^2)^\frac{p-4}{2}\,\left|u^\varepsilon_{x_i}\right|^2\,\left|u^\varepsilon_{x_i\,x_j}\right|^2 \,dx=0.
\end{split}
\]
The previous can be rewritten as
\[
\sum_{i=1}^2\int_{\left\{x\in B_r\, :\, u_{x_j}^\varepsilon \le C\right\}} (\varepsilon+|u^\varepsilon_{x_i}|^2)^\frac{p-4}{2}\,(\varepsilon+(p-1)\,|u_{x_i}^\varepsilon|^2)\,\left|u^\varepsilon_{x_j\,x_i}\right|^2\,dx=0,
\]
which in turn implies
\[
\sum_{i=1}^2\int_{\left\{x\in B_r\, :\, u_{x_j}^\varepsilon \le C\right\}} \left|u^\varepsilon_{x_j\,x_i}\right|^2\,dx=0,\quad \mbox{ i.\,e. }\int_{\left\{x\in B_r\, :\, u_{x_j}^\varepsilon \le C\right\}} \left|\nabla u^\varepsilon_{x_j}\right|^2\,dx=0.
\]
From this identity, we get that the Sobolev function
\[
(C-u^\varepsilon_{x_j})_+,
\]
is constant in $B_r$ and thanks to the fact that $u^\varepsilon_{x_j}\ge C$ on $\partial B_r$, we get
\[
(C-u^\varepsilon_{x_j})_+=0,\qquad \mbox{ in } B_r,
\]
as desired.
\end{proof}

\begin{lm}
\label{lm:traccebis}
Let $B_r\Subset B$. With the notation above, there exists a sequence \(\{\varepsilon_k\}_{k\in \mathbb{N}}\) such that for almost every $s\in[0,r]$, we have
\[
\lim_{k\to +\infty}\|u^{\varepsilon_k}_{x_j}-U_{x_j}\|_{L^\infty(\partial B_s)}=0,\qquad j=1,2.
\]
\end{lm}
\begin{proof}
Observe that $\{u^\varepsilon_{x_j}-U_{x_j}\}_{0<\varepsilon<\varepsilon_0}$ weakly converges to \(0\) in $W^{1,2}(B_r)$, thanks to Proposition \ref{prop:convergence_uep1p2bis}. The proof then runs similarly to that of Lemma \ref{lm:tracce}.
\end{proof}

\section{Caccioppoli inequalities}
\label{sec:3}

\subsection{The case $p>2$}
One of the key ingredients in the proof of Theorem \ref{teo:theorem-continuity-N2} for $p>2$ is the following ``weird'' Caccioppoli inequality for the gradient of the local minimizer $U$. Observe that the inequality contains quantities like the product of different components of $\nabla U$.
\begin{prop}
\label{prop:mixing}
Let \(\Phi: \mathbb{R}\to \mathbb{R}\) be a \(C^2\) function such that \(\Phi\,\Phi''\geq 0\) and \(\zeta : \mathbb{R}\to \mathbb{R}^+\) be a nonnegative convex function. For every $B \Subset \Omega$, every \(\eta \in C^{\infty}_0(B)\) and every \(j, k\in\{1, 2\}\), 
\begin{equation}
\label{eq_Cacciopoli_inequality_nonsmooth}
\begin{split}
\sum_{i=1}^2 \int \left|\left(|U_{x_i}|^\frac{p-2}{2}\,U_{x_i}\right)_{x_k}\right|^2\, [\Phi'(U_{x_k})]^2\, \zeta(U_{x_j})\,\eta^2 \,dx
&\leq 
C \left( \sum_{i=1}^2 \int |U_{x_i}|^{p-2}\,\Phi(U_{x_k})^4\, |\eta_{x_i}|^2\,dx\right)^{\frac{1}{2}}\\
&\times\left( \sum_{i=1}^2 \int |U_{x_i}|^{p-2}\, \zeta(U_{x_j})^2\, |\eta_{x_i}|^2\,dx\right)^{\frac{1}{2}}.
\end{split}
\end{equation}
\end{prop}
\begin{proof}
By a standard approximation argument, one can assume \(\zeta\) to be a smooth function.
We fix \(\varepsilon>0\) and we take as above $u^\varepsilon$ the minimizer of \eqref{approximation}, subject to the boundary condition $u^\varepsilon-U^\varepsilon \in W^{1,p}_0(B)$. We divide the proof in two parts: we first show \eqref{eq_Cacciopoli_inequality_nonsmooth} for $u^\varepsilon$ and then prove that we can take the limit.
\vskip.2cm\noindent
{\bf Caccioppoli for $u^\varepsilon$.} We plug into \eqref{Differentiation_Euler_equation_varepsilon} the test function
\[
\varphi=\Psi(u^\varepsilon_{x_k})\,\zeta(u^\varepsilon_{x_j})\,\eta^2,\qquad \mbox{ where }\Psi(t)=\Phi(t)\,\Phi'(t),
\]  
where $\eta$ is as in the statement. In order to simplify the notation, we write \(u\) in place of \(u^\varepsilon\) in what follows.
Since 
\[
\varphi_{x_i}=u_{x_k\, x_i}\Psi'(u_{x_k})\,\zeta(u_{x_j})\,\eta^2 + \Psi(u_{x_k})\,\left(\zeta(u_{x_j})\right)_{x_i}\eta^2 + 2\,\eta\,\eta_{x_i}\,\Psi(u_{x_k})\,\zeta(u_{x_j}),
\]  
we obtain
\begin{equation}
\label{debut}
\begin{split}
\sum_{i=1}^{2}&\int(|u_{x_i}|^{p-2}+\varepsilon)\,u_{x_i\,x_k}^2\,\Psi'(u_{x_k})\,\zeta(u_{x_j})\,\eta^2\,dx\\
&=-\sum_{i=1}^{2}\int(|u_{x_i}|^{p-2}+\varepsilon)\,u_{x_i x_k}\,\Psi(u_{x_k})\,\left(\zeta(u_{x_j})\right)_{x_i}\,\eta^2\,dx\\
& -2\sum_{i=1}^{2}\int(|u_{x_i}|^{p-2}+\varepsilon)\,u_{x_i x_k}\,\Psi(u_{x_k})\,\zeta(u_{x_j})\,\eta\,\eta_{x_i}\,dx.
\end{split}
\end{equation}
For the second term in the right-hand side,  the Young inequality implies
\[
\begin{split}
2\,\int (|u_{x_i}|^{p-2}+\varepsilon)\,u_{x_i x_k}\,\Psi(u_{x_k})\,\zeta(u_{x_j})\,\eta\,\eta_{x_i}\,dx\leq &\frac{1}{2}\int (|u_{x_i}|^{p-2}+\varepsilon)\,u_{x_i x_k}^2\,\Phi'(u_{x_k})^2\,\zeta(u_{x_j})\,\eta^2\,dx\\
&+ 2\,\int(|u_{x_i}|^{p-2}+\varepsilon)\,\Phi(u_{x_k})^2\,\zeta(u_{x_j})\,\eta_{x_i}^2\,dx,
\end{split}
\]
where we used the definition of $\Psi$. The first term can be absorbed in the left-hand side of \eqref{debut}, thanks to the fact that 
\[
\Psi'=(\Phi\,\Phi')'=\Phi'^2 + \Phi\,\Phi''\geq \Phi'^2.
\]
Hence, for the moment we have obtained
\begin{equation}
\begin{split}
\label{eq1317}
\sum_{i=1}^{2}&\int(|u_{x_i}|^{p-2}+\varepsilon)\,u_{x_i x_k}^2\,\Phi'(u_{x_k})^2\,\zeta(u_{x_j})\,\eta^2\,dx\\&\leq 2 \sum_{i=1}^{2}\int (|u_{x_i}|^{p-2}+\varepsilon)\,|u_{x_i x_k}|\,|\Psi(u_{x_k})|\,\left|\left(\zeta(u_{x_j})\right)_{x_i}\right|\,\eta^2\,dx\\
&+4\, \sum_{i=1}^{2}\int (|u_{x_i}|^{p-2}+\varepsilon)\,\Phi(u_{x_k})^2\zeta(u_{x_j})\,\eta_{x_i}^2\,dx.
\end{split}
\end{equation}
In the particular case when \(\zeta\equiv 1\), we observe for later use that
\begin{equation}
\begin{split}
\label{eq_old_Cacciopoli}
\sum_{i=1}^{2}\int(|u_{x_i}|^{p-2}+\varepsilon)\,\left|\left(\Phi(u_{x_k})\right)_{x_i}\right|^2\,\eta^2\,dx
&=\sum_{i=1}^{2}\int (|u_{x_i}|^{p-2}+\varepsilon)\,u_{x_i x_k}^2\,\Phi'(u_{x_k})^2\,\eta^2\,dx\\
&\leq 4 \sum_{i=1}^{2}\int(|u_{x_i}|^{p-2}+\varepsilon)\,\Phi(u_{x_k})^2\,\eta_{x_i}^2\,dx.
\end{split}
\end{equation}
We go back to \eqref{eq1317}. By H\"older inequality, we can estimate the last term of the right-hand side:
\begin{equation}
\label{eq1331}
\begin{split}
\sum_{i=1}^{2}\int(|u_{x_i}|^{p-2}+\varepsilon)\,\Phi(u_{x_k})^2\,\zeta(u_{x_j})\,\eta_{x_i}^2\,dx&\leq \left( \sum_{i=1}^{2}\int (|u_{x_i}|^{p-2}+\varepsilon)\,\Phi(u_{x_k})^4\,\eta_{x_i}^2\,dx\right)^{\frac{1}{2}}\\
&\times\left(\sum_{i=1}^{2}\int(|u_{x_i}|^{p-2}+\varepsilon)\,\zeta(u_{x_j})^2\,\eta_{x_i}^2\,dx\right)^{\frac{1}{2}}.
\end{split}
\end{equation}
In a similar fashion, for the first term in the right-hand side of \eqref{eq1317}, we have
\begin{equation}
\begin{split}
\label{eq1340}
\sum_{i=1}^{2}&\int (|u_{x_i}|^{p-2}+\varepsilon)\,|u_{x_i x_k}|\,|\Psi(u_{x_k})|\,\left|\left(\zeta(u_{x_j})\right)_{x_i}\right|\,\eta^2\,dx\\
&\leq \left(\sum_{i=1}^{2} \int (|u_{x_i}|^{p-2}+\varepsilon)\,u_{x_i x_k}^2\,\Psi(u_{x_k})^2\eta^2\,dx\right)^{\frac{1}{2}}\,\left(\sum_{i=1}^{2} \int (|u_{x_i}|^{p-2}+\varepsilon)\,\left|\left(\zeta(u_{x_j})\right)_{x_i}\right|^2\,\eta^2\,dx\right)^{\frac{1}{2}}\\ 
&= \frac{1}{2}\,\left(\sum_{i=1}^{2} \int(|u_{x_i}|^{p-2}+\varepsilon)\left|\left(\Phi(u_{x_k})^2\right)_{x_i}\right|^2
\eta^2\,dx\right)^{\frac{1}{2}}
\left(\sum_{i=1}^{2} \int (|u_{x_i}|^{p-2}+\varepsilon)\,\left|\left(\zeta(u_{x_j})\right)_{x_i}\right|^2\,\eta^2\,dx\right)^{\frac{1}{2}}.
\end{split}
\end{equation}
In the last equality, we have used the fact that 
\[
u_{x_i x_k}^2\,\Psi(u_{x_k})^2=\frac{1}{4}\,\left(\left(\Phi(u_{x_k})^2\right)_{x_i}\right)^2.
\]
It follows from \eqref{eq1317}, \eqref{eq1331} and \eqref{eq1340} that
\[
\begin{split}
\sum_{i=1}^{2}&\int(|u_{x_i}|^{p-2}+\varepsilon)u_{x_i x_k}^2\,\Phi'(u_{x_k})^2\zeta(u_{x_j})\eta^2
\\ 
&\leq \left(\sum_{i=1}^{2} \int(|u_{x_i}|^{p-2}+\varepsilon)\,\left|\left(\Phi(u_{x_k})^2\right)_{x_i}\right|^2
\eta^2\,dx\right)^{\frac{1}{2}}
\left(\sum_{i=1}^{2} \int (|u_{x_i}|^{p-2}+\varepsilon)\,\left|\left(\zeta(u_{x_j})\right)_{x_i}\right|^2\,\eta^2\,dx\right)^{\frac{1}{2}}\\
&+4\, \left(\sum_{i=1}^{2} \int(|u_{x_i}|^{p-2}+\varepsilon)\,\Phi(u_{x_k})^4\,\eta_{x_i}^2\,dx\right)^{\frac{1}{2}}
\,\left(\sum_{i=1}^{2} \int(|u_{x_i}|^{p-2}+\varepsilon)\,\zeta(u_{x_j})^2\,\eta_{x_i}^2\,dx\right)^{\frac{1}{2}}.
\end{split}
\]
By \eqref{eq_old_Cacciopoli} with\footnote{Observe that $\Phi^2$ still verifies $\Phi^2\,(\Phi^2)''\ge 0$. Indeed, $(\Phi^2)''=2\,(\Phi')^2+2\,\Phi\,\Phi''\ge 0$, by hypothesis.} \(\Phi^2\) in place of \(\Phi\), one has 
\[
\sum_{i=1}^{2} \int(|u_{x_i}|^{p-2}+\varepsilon)\left|\left(\Phi(u_{x_k})^2\right)_{x_i}\right|^2
\eta^2\,dx
\leq 4\,\sum_{i=1}^2\int (|u_{x_i}|^{p-2}+\varepsilon)\,\Phi(u_{x_k})^4\,\eta_{x_i}^2\,dx.
\]
Similarly, by using \eqref{eq_old_Cacciopoli} with \(\zeta\) in place of \(\Phi\) and \(j\) in place  of \(k\),
\[
\sum_{i=1}^2\int (|u_{x_i}|^{p-2}+\varepsilon)\,\left|\left(\zeta(u_{x_j})\right)_{x_i}\right|^2\,\eta^2 \,dx\leq 4\,\sum_{i=1}^2\int (|u_{x_i}|^{p-2}+\varepsilon)\,\zeta(u_{x_j})^2\,\eta_{x_i}^2\,dx.
\]
Hence, we have obtained
\[
\begin{split}
\sum_{i=1}^{2}&\int(|u_{x_i}|^{p-2}+\varepsilon)\,u_{x_i x_k}^2\,\Phi'(u_{x_k})^2\,\zeta(u_{x_j})\,\eta^2\,dx\\
&\leq C\, \left(\int\sum_{i=1}^2 (|u_{x_i}|^{p-2}+\varepsilon)\,\Phi(u_{x_k})^4\,\eta_{x_i}^2\,dx\right)^{\frac{1}{2}}\,\left(\int\sum_{i=1}^2 (|u_{x_i}|^{p-2}+\varepsilon)\,\zeta(u_{x_j})^2\,\eta_{x_i}^2\,dx\right)^{\frac{1}{2}},
\end{split}
\]
for some universal constant $C>0$.
We now observe that 
\[
(|u_{x_i}|^{p-2}+\varepsilon)\,u_{x_i x_k}^2 \geq |u_{x_i}|^{p-2}u_{x_i x_k}^2=\frac{4}{p^2}\,\left|\left(|u_{x_i}|^\frac{p-2}{2}\,u_{x_i}\right)_{x_k}\right|^2,
\] 
thus, by restoring the original notation \(u^{\varepsilon}\), we get
\begin{equation}
\label{prefinal}
\begin{split}
\sum_{i=1}^{2}&\int \left|\left(|u^\varepsilon_{x_i}|^\frac{p-2}{2}\,u^\varepsilon_{x_i}\right)_{x_k}\right|^2\,\Phi'(u_{x_k}^{\varepsilon})^2\,\zeta(u_{x_j}^\varepsilon)\,\eta^2\,dx\\
&\leq C\, \left(\sum_{i=1}^2\int (|u_{x_i}^\varepsilon|^{p-2}+\varepsilon)\,\Phi(u_{x_k}^\varepsilon)^4\,\eta_{x_i}^2\,dx\right)^{\frac{1}{2}}\left(\sum_{i=1}^2\int  (|u_{x_i}^\varepsilon|^{p-2}+\varepsilon)\,\zeta(u_{x_j}^\varepsilon)^2\,\eta_{x_i}^2\,dx\right)^{\frac{1}{2}}.
\end{split}
\end{equation}
{\bf Passing to the limit $\varepsilon\to 0$.} 
By Lemma \ref{lm:regularity}, for every $B_r\Subset B$ the gradient \(\nabla u^\varepsilon\) is uniformly bounded in \(L^{\infty}(B_r)\). Moreover, by Proposition \ref{lm_convergence_ueps}, up to a subsequence (we do not relabel), it converges almost everywhere to \(\nabla U\). 
By recalling that $\eta$ has compact support in $B$, then the Dominated Convergence Theorem implies that the right-hand side of \eqref{prefinal} converges to the corresponding quantity with \(U\) in place of \(u^{\varepsilon}\) and \(\varepsilon=0\). 
\par
As for the left-hand side, we use the fact that for a subsequence (still denoted by \(u^{\varepsilon}\))
\[
\left\|\Phi'(u_{x_k}^{\varepsilon})\,\sqrt{\zeta(u_{x_j}^\varepsilon)}\,\eta\right\|_{L^\infty(\mathrm{spt}(\eta))}\le C,\qquad \Phi'(u_{x_k}^{\varepsilon})\,\sqrt{\zeta(u_{x_j}^\varepsilon)}\,\eta\to \Phi'(U_{x_k})\,\sqrt{\zeta(U_{x_j})}\,\eta,\ \mbox{ a.\,e.},
\] 
and that 
\[
|u^\varepsilon_{x_i}|^\frac{p-2}{2}\,u^\varepsilon_{x_i}\ \mbox{ weakly converges in $W^{1,2}(\mathrm{spt}(\eta))$ to }\ |U_{x_i}|^\frac{p-2}{2}\,U_{x_i}.
\]
still by Proposition \ref{lm_convergence_ueps}. Hence, we can infer weak convergence in \(L^{2}(\mathrm{spt}(\eta))\) of 
\[
\big(|u^\varepsilon_{x_i}|^\frac{p-2}{2}\,u^\varepsilon_{x_i}\big)_{x_k} \,\Phi'(u_{x_k}^{\varepsilon})\,\sqrt{\zeta(u_{x_j}^\varepsilon)}\,\eta.
\]
Finally, by semicontinuity of the norm with respect to weak convergence, one gets
\[
\int \left|\left(|U_{x_i}|^\frac{p-2}{2}\,U_{x_i}\right)_{x_k}\right|^2\,\Phi'(U_{x_k})^2\,\zeta(U_{x_j})\,\eta^2\,dx
\leq \liminf_{\varepsilon\to 0}\int \left|\left(|u^\varepsilon_{x_i}|^\frac{p-2}{2}\,u^\varepsilon_{x_i}\right)_{x_k}\right|^2\,\Phi'(u_{x_k}^{\varepsilon})^2\,\zeta(u_{x_j}^\varepsilon)\,\eta^2\,dx.
\]
This yields the desired estimate \eqref{eq_Cacciopoli_inequality_nonsmooth} for $U$.
\end{proof}

\subsection{The case $1<p\le 2$}

In this case, the Caccioppoli inequality we need is more standard.
\begin{prop}\label{cacciople2}
Let $\zeta:\mathbb{R}\to \mathbb{R}$ be a $C^1$ monotone function. For every $B\Subset \Omega$, every \(\eta \in C^{\infty}_0(B)\) and every $j=1,2$ we have
\begin{equation}
\label{uniformesobsub2limit}
\begin{split}
\sum_{i=1}^2\int_{\{U_{x_i}\not=0\}} & |U_{x_i}|^{p-2}\,\left|\left(Z(U_{x_j})\right)_{x_i}\right|^2\,\eta^2\,dx\\
&\leq C\int |\nabla U|^{p-1}\,\Big(|\nabla U|\,|\zeta'(U_{x_j})|+|\zeta(U_{x_j})|\Big)\,\Big(|\nabla \eta|^2+|D^2\eta|\Big)\,dx,
\end{split}
\end{equation}
where $Z:\mathbb{R}\to\mathbb{R}$ is the $C^1$ function defined by
\begin{equation}
\label{zetona}
Z(t)=\int_0^t \sqrt{|\zeta'(s)|}\,ds.
\end{equation}
\end{prop}
\begin{proof}
We fix \(\varepsilon>0\) and we take as above $u^\varepsilon$ the minimizer of \eqref{approximationsub}, subject to the boundary condition $u^\varepsilon-U^\varepsilon \in W^{1,p}_0(B)$. Then by Proposition \ref{lm:caccioespilon2}, we have
\[
\begin{split}
\sum_{i=1}^2\int & (\varepsilon+|u^\varepsilon_{x_i}|^2)^\frac{p-2}{2}\,|\zeta'(u^\varepsilon_{x_j})|\,\left|u^\varepsilon_{x_j\,x_i}\right|^2\,\eta^2\,dx\\
& \le C\int (\varepsilon+|\nabla u^\varepsilon|^2)^\frac{p}{2}\,|\zeta'(u^\varepsilon_{x_j})|\,|\nabla \eta|^2\,dx\\
&+C\int (\varepsilon+|\nabla u^\varepsilon|^2)^\frac{p-1}{2}\,|\zeta(u^\varepsilon_{x_j})|\,\Big(|\nabla \eta|^2+|D^2\eta|\Big)\,dx,
\end{split}
\]
for some \(C=C(p)>0\).
Since \(p<2\), 
\[
(\varepsilon+|u^\varepsilon_{x_i}|^2)^\frac{p-2}{2}\,|\zeta'(u^\varepsilon_{x_j})|\,\left|u^\varepsilon_{x_j\,x_i}\right|^2\,\eta^2 \geq \left( \left( (\varepsilon + |u^{\varepsilon}_{x_j}|^2)^{\frac{p-2}{4}} u^{\varepsilon}_{x_i}\right)_{x_j}\sqrt{|\zeta'(u^{\varepsilon}_{x_j})|}\eta\right)^2.
\]
Hence,
\begin{equation}\label{eq1168}
\begin{split}
\sum_{i=1}^2\int &\left( \left( (\varepsilon + |u^{\varepsilon}_{x_j}|^2)^{\frac{p-2}{4}} u^{\varepsilon}_{x_i}\right)_{x_j}\sqrt{|\zeta'(u^{\varepsilon}_{x_j})|}\eta\right)^2\\
& \le C\int (\varepsilon+|\nabla u^\varepsilon|^2)^\frac{p}{2}\,|\zeta'(u^\varepsilon_{x_j})|\,|\nabla \eta|^2\,dx\\
&+C\int (\varepsilon+|\nabla u^\varepsilon|^2)^\frac{p-1}{2}\,|\zeta(u^\varepsilon_{x_j})|\,\Big(|\nabla \eta|^2+|D^2\eta|\Big)\,dx,
\end{split}
\end{equation}

In order to pass to the limit as $\varepsilon$ goes to $0$, we observe that by Lemma \ref{lm:lemma1p2}, for every $B_r\Subset B$ the gradient \(\nabla u^\varepsilon\) is uniformly bounded in \(L^{\infty}(B_r)\). Moreover, by Proposition \ref{prop:convergence_uep1p2bis} it converges almost everywhere to \(\nabla U\) (up to a subsequence). 
By recalling that $\eta$ has compact support in $B$, then the Dominated Convergence Theorem implies that the right-hand side of the above inequality converges to the corresponding quantity with \(U\) in place of \(u^{\varepsilon}\) and \(\varepsilon=0\).
\par
As for the left-hand side, we observe that by Proposition \ref{prop:convergence_uep1p2bis}
\[
(\varepsilon+|u^{\varepsilon}_{x_i}|^2)^{\frac{p-2}{4}}\,u^{\varepsilon}_{x_i}\quad \mbox{ weakly converges in $W^{1,2}(\mathrm{spt}(\eta))$ to }\quad |U_{x_i}|^{\frac{p-2}{2}}\,U_{x_i},
\] 
and (up to a subsequence),
\[
\left\|\sqrt{|\zeta'(u^\varepsilon_{x_j})|}\,\eta\right\|_{L^\infty(\mathrm{spt}(\eta))}\le C,\qquad \sqrt{|\zeta'(u^\varepsilon_{x_j})|}\,\eta\to \sqrt{|\zeta'(U_{x_j})|}\,\eta \quad \mbox{a.\,e.}
\]
Thus as in the case $p>2$, we can infer weak convergence in $L^2(\mathrm{spt}(\eta))$ of
\[
\left((\varepsilon+|u^\varepsilon_{x_i}|^2)^\frac{p-2}{4} u^{\varepsilon}_{x_i}\right)_{x_j}\,\sqrt{|\zeta'(u^\varepsilon_{x_j})|}\,\eta.
\]
By the same semicontinuity argument as before, we get
\[
\liminf_{\varepsilon\to 0}\sum_{i=1}^2\int  \left( \left( (\varepsilon + |u^{\varepsilon}_{x_j}|^2)^{\frac{p-2}{4}} u^{\varepsilon}_{x_i}\right)_{x_j}\sqrt{|\zeta'(u^{\varepsilon}_{x_j})|}\eta\right)^2\,dx
\ge \sum_{i=1}^2\int \left|\left(|U_{x_i}|^\frac{p-2}{2} U_{x_i}\right)_{x_j}\,\sqrt{|\zeta'(U_{x_j})|}\,\eta\right|^2\,dx.
\]
The right-hand side is greater than or equal to
\[
\sum_{i=1}^2\int_{\{U_{x_i}\not=0\}} \left|\left(|U_{x_i}|^\frac{p-2}{2} U_{x_i}\right)_{x_j}\right|^2\,|\zeta'(U_{x_j})|\,\eta^2\,dx
= \frac{p^2}{4}\sum_{i=1}^2\int_{\{U_{x_i}\not=0\}} \left||U_{x_i}|^\frac{p-2}{2} U_{x_i\,x_j}\right|^2\,|\zeta'(U_{x_j})|\,\eta^2\,dx.
\]
The last equality follows from \eqref{eq_chain_rule_weird}. Now, applying the standard chain rule for the \(C^{1}\) function \(Z\) defined in \eqref{zetona} (remember also that $U_{x_j}\in W^{1,2}_{\rm loc}(\Omega)\cap L^\infty_{\rm loc}(\Omega)$) yields
\[
\begin{split}
\liminf_{\varepsilon\to 0}\sum_{i=1}^2\int &\left( \left( (\varepsilon + |u^{\varepsilon}_{x_j}|^2)^{\frac{p-2}{4}}\,u^{\varepsilon}_{x_i}\right)_{x_j}\,\sqrt{|\zeta'(u^{\varepsilon}_{x_j})|}\,\eta\right)^2 \,dx\\
&\ge \frac{p^2}{4}\,\sum_{i=1}^2 \int_{\{U_{x_i}\not=0\}} |U_{x_i}|^{p-2}\,\left|\left(Z(U_{x_j})\right)_{x_i}\right|^2\,\eta^2\,dx.
\end{split}
\]
In view of \eqref{eq1168}, this completes the proof.
\end{proof}

\section{Decay estimates for a nonlinear function\\ of the gradient for $p>2$}
\label{sec:4}

We already know from Theorem \ref{teo:apriori} that 
\[
|U_{x_j}|^\frac{p-2}{2}\,U_{x_j}\in W^{1,2}_{\rm loc}(\Omega)\cap L^\infty_{\rm loc}(\Omega).
\] 
This nonlinear function of the gradient of $U$ will play a crucial role in the sequel, for the case $p> 2$. Thus we introduce the expedient notation
\[
v_j=|U_{x_j}|^{\frac{p-2}{2}}\,U_{x_j},\qquad j=1,2.
\]
For every $B_R\Subset \Omega$, we will also use the following notation:
\begin{equation}
\label{notations}
m_j=\inf_{B_R} v_j,\qquad V_j = v_j-m_j, \qquad M_j=\sup_{B_R} V_j = \osc_{B_R}v_j,\qquad j=1,2,
\end{equation}
and
\begin{equation}
\label{notations2}
L_R=1+\|\nabla U\|_{L^\infty(B_R)}.
\end{equation}
\subsection{A De Giorgi-type Lemma}
We first need the following result on the decay of the oscillation of $v_j$.
This is the analogue of \cite[Lemma 4]{Santambrogio-Vespri}. As explained in the Introduction, our operator is much more degenerate then the one considered in \cite{Santambrogio-Vespri}, thus the proof has to be completely recast. We crucially rely on the Caccioppoli inequality of Proposition \ref{prop:mixing}.
\begin{lm}
\label{lemmaDeGiorgi}
Let \(B_R\Subset \Omega\) and \(0<\alpha <1\). By using the notation in \eqref{notations} and \eqref{notations2}, there exists a constant \(\nu=\nu(p,\alpha,L_R)>0\) such that if
\[
\Big|\{V_j>(1-\alpha)\,M_j\} \cap B_R\Big| \leq \nu\, M_{j}^{2\,p+4\left(1-\frac{2}{p}\right)}|B_R|,
\]
then
\[
0\leq V_j \leq \left(1-\frac{\alpha}{2}\right)\, M_j,\ \mbox{ on } B_{\frac{R}{2}}.
\]
\end{lm}
\begin{proof}
We first observe that if $M_j=0$, then $V_j$ identically vanishes in $B_R$ and there is nothing to prove. Thus, we can assume that $M_j>0$.
\par
For \(n\geq 1\), we set
\[
k_n=M_j\left(1-\frac{\alpha}{2} - \frac{\alpha}{2^n} \right), \qquad R_n=\frac{R}{2} + \frac{R}{2^n}, \qquad A_n=\{V_j>k_n\}\cap B_{R_n},
\]
where the ball $B_{R_n}$ is concentric with $B_R$.
Let \(\theta_n\) be a smooth cut-off function such that
\[
0\le \theta_n\le 1,\qquad\theta_n  \equiv 1 \textrm{ on } B_{R_{n+1}}, \qquad \theta_n \equiv 0 \textrm{ on } \mathbb{R}^2\setminus B_{R_n}, \qquad |\nabla \theta_n| \leq C\,\frac{2^n}{R}.
\]
Recalling the definition \eqref{gq} of $g_q$, we then set for every \(n\geq 1\)
\begin{equation}
\label{betan}
\beta_n=g_\frac{p-2}{2}^{-1}(m_j+k_n)=|m_j+k_n|^\frac{2-p}{p}\,(m_j+k_n).
\end{equation}
We start from \eqref{eq_Cacciopoli_inequality_nonsmooth} with the choices 
\[
\Phi(t)=t,\qquad \zeta(t)=(t-\beta_n)_{+}^2\qquad \mbox{ and }\qquad \eta=\theta_n.
\] 
Observe that 
\[
\zeta(U_{x_j}) = (U_{x_j}-\beta_n)_{+}^2 >0\qquad \Longleftrightarrow\qquad V_j >k_n,
\] 
and also\footnote{In the second inequality we use that $t\mapsto g^{-1}_{(p-2)/2}(t)$ is $2/p-$H\"older continuous.}
\begin{equation}\label{eq1448}
\begin{split}
0\leq \zeta(U_{x_j})&\leq \left| g_{\frac{p-2}{2}}^{-1}(v_j) - g_{\frac{p-2}{2}}^{-1}(m_j+k_n) \right|^2\\
&\leq C\,|v_j-m_j-k_n|^{\frac{4}{p}} \leq C M_{j}^{\frac{4}{p}},\qquad \mbox{ a.\,e. on } B_{R_n}. 
\end{split}
\end{equation}
We then obtain using \eqref{eq1448} and  the definition of \(A_n\),
\[
\begin{split}
\sum_{i=1}^{2}\int \left|(v_{i})_{x_k}\right|^2\,\zeta(U_{x_j})\,\theta_n^2
&\leq 
C \left(\sum_{i=1}^2\int |U_{x_i}|^{p-2}\,|U_{x_k}|^4\,\left|(\theta_n)_{x_i}\right|^2\,dx\right)^{\frac{1}{2}}\left(\sum_{i=1}^2\int |U_{x_i}|^{p-2}\,\zeta(U_{x_j})^2\,\left|(\theta_n)_{x_i}\right|^2\,dx\right)^{\frac{1}{2}}
\\ 
&\leq C\, L^{p}_R\,M_{j}^{\frac{4}{p}} \left(\int_{B_{R_n}}|\nabla \theta_n|^2\right)^{\frac{1}{2}} \left(\int_{A_n}|\nabla \theta_n|^2\right)^{\frac{1}{2}}.
\end{split}
\]
In view of the properties of \(\theta_n\), it follows that
\begin{align*}
\sum_{i=1}^{2}\int\left|(v_{i})_{x_k}\right|^2\,\zeta(U_{x_j})\,\theta_{n}^2\,dx
&\leq  C\,L^{p}_R\,M_{j}^{\frac{4}{p}} \left(\frac{2^n}{R} \right)^{2} |B_{R_n}\setminus B_{R_{n+1}}|^{\frac{1}{2}}|A_n|^{\frac{1}{2}}\\
&\leq C\,4^n\,L^{p}_R\,M_{j}^{\frac{4}{p}}\, \frac{|A_n|^{\frac{1}{2}}}{R},
\end{align*}
for some $C=C(p)>0$.
Here, we have used that
\[
|B_{R_n}\setminus B_{R_{n+1}}|=\pi\left(R_n^2 - R_{n+1}^2\right)=\pi\,(R_n-R_{n+1})\,(R_n+R_{n+1})\leq \frac{R}{2^{n+1}}\,2\,\pi\,R= \pi\, \frac{R^2}{2^n}.
\]
In the left-hand side, we only keep the term \(i=j\) and use that by Lemma \ref{lm:schwarz2}
\[
(v_{j})_{x_k}\,\sqrt{\zeta(U_{x_j})} = \left(F(U_{x_j})\right)_{x_k},
\]
where
\[
F(t) =\frac{p}{2}\, \int_{\beta_n}^{t}|s|^{\frac{p-2}{2}}\,\sqrt{\zeta(s)}\,ds =\frac{p}{2}\, \int_{\beta_n}^{t}|s|^{\frac{p-2}{2}}(s-\beta_n)_+\,ds, \quad t\in \mathbb{R}.
\]
We thus obtain
\[
\int\left|(F(U_{x_j}))_{x_k}\right|^2 \theta_n^2\,dx \leq  C\,4^n\,L^{p}_R\,M_{j}^{\frac{4}{p}}\, \frac{|A_n|^{\frac{1}{2}}}{R}.
\]
Summing over \(k=1,2\), this yields an estimate for the gradient of $F(U_{x_j})$, i.\,e.
\begin{equation}
\label{caccio1}
\int |\nabla (F(U_{x_j}))|^2\, \theta_{n}^2\,dx \leq  C\,4^n\, L^{p}_R\,M_{j}^{\frac{4}{p}}\,\frac{|A_n|^{\frac{1}{2}}}{R}. 
\end{equation}
Since \(m_j\leq m_j+k_n\leq m_j + M_j = \sup_{B_R}v_j\) and  by definition of \(L_R\),  \(|m_j+k_n|\leq L^{p/2}_R\). Hence, by definition of \(\beta_n\), see \eqref{betan},
\begin{equation}
\label{betaL}
|\beta_n|\leq L_R.
\end{equation}
By keeping this in mind and using Lemma \ref{lemma_F} below, 
\[
0\leq F(U_{x_j}) \leq C \left( |U_{x_j}|^{\frac{p-2}{2}} + |\beta_n|^\frac{p-2}{2}\right)\,(U_{x_j}-\beta_n)^{2}_+
\leq C\, L_R^{\frac{p-2}{2}}\, (U_{x_j}-\beta_n)^{2}_+.
\]
This implies that \(F(U_{x_j})=0\) on \(B_{R_n}\setminus A_n\) and also that
\[
0\leq F(U_{x_j}) \leq C\,  L^{\frac{p-2}{2}}_R\,\zeta(U_{x_j}) \leq C\, L^{\frac{p-2}{2}}_R\,M_{j}^{\frac{4}{p}},
\]
for some $C=C(p)>0$.
In the last inequality, we have used \eqref{eq1448}.
Hence,
\begin{equation}
\label{caccio2}
\begin{split}
\int|\nabla \theta_n|^2 (F(U_{x_j}))^2\,dx &\leq C\,L^{p-2}_R\,M_{j}^{\frac{8}{p}}\,\int_{A_{n}}|\nabla \theta_n|^2\,dx\\
&\leq C\,4^n\,L^{p-2}_R\,M_{j}^{\frac{8}{p}}\,\frac{|A_n|}{R^2}\leq C\,4^n\,L^{p}_R\,M_{j}^{\frac{4}{p}}\,\frac{|A_n|^{\frac{1}{2}}}{R},
\end{split}
\end{equation}
where in the last inequality we used that \(|A_n|^{1/2}\leq \sqrt{\pi}\,R\) and \(M_j\leq 2 L^{p/2}_R\). By adding \eqref{caccio1} and \eqref{caccio2}, with some simple manipulations we get 
\[
\int_{B_{R_n}}|\nabla (F(U_{x_j})\,\theta_n)|^2 \leq C\,4^n\, L^{p}_R\,M_{j}^{\frac{4}{p}}\, \frac{|A_n|^{\frac{1}{2}}}{R},
\]
where as usual $C=C(p)>0$.
We now rely on the following Poincar\'e inequality\footnote{For every bounded open set \(\Omega\subset \mathbb{R}^2\), the Sobolev embedding \(W^{1,1}_0(\Omega)\hookrightarrow L^{2}(\Omega)\) implies that for every \(f\in W^{1,2}_0(\Omega)\), \[\int |f|^2\,dx \leq C \,\left(\int |\nabla f|\,dx\right)^2=C\,\left(\int_{\{f\not=0\}} |\nabla f|\,dx\right)^2 \leq C\,|\{x : f(x)\not=0\}|\,\int |\nabla f|^2\,dx, \] where \(C\) is a universal constant.} for the function $F(U_{x_j})\,\theta_n\in W^{1,2}_0(B_{R_n})$
\[
\left|\{x\in B_{R_n}\,:\, F(U_{x_j})\,\theta_n>0\}\right|\, \int_{B_{R_n}}|\nabla(F(U_{x_j})\,\theta_n)|^2\,dx \geq c\, \int_{B_{R_n}}|F(U_{x_j})\,\theta_n|^2\,dx.
\]
Since \(\theta_n\equiv 1\) on \(B_{R_{n+1}}\) and by construction
\[
|A_n| \geq \left|\{F(U_{x_j})\,\theta_n>0\}\right|,
\]
one gets
\[
\int_{B_{R_{n+1}}}|F(U_{x_j})|^2\,dx \leq C\,\frac{4^n\, L^{p}_R\,M_{j}^{\frac{4}{p}}}{R}\, |A_n|^{\frac{3}{2}},
\]
for some $C=C(p)>0$.
By using that \(F\) is increasing on \([\beta_n,+\infty)\) and
\[
A_{n+1} =\{V_j>k_{n+1}\}\cap B_{R_{n+1}}=\{U_{x_j}>\beta_{n+1}\}\cap B_{R_{n+1}},
\]
we obtain
\[
\int_{B_{R_{n+1}}}|F(U_{x_j})|^2 \,dx\geq \int_{A_{n+1}}\, |F(U_{x_j})|^2\,dx \geq |A_{n+1}|\,F(\beta_{n+1})^2.
\]
This gives
\begin{equation}\label{eq1529}
|A_{n+1}|\,F(\beta_{n+1})^2 \leq C\, \frac{4^n\,L^{p}_R\,M_{j}^{\frac{4}{p}}}{R} |A_n|^{\frac{3}{2}}.
\end{equation}
We now use the lower bound of Lemma \ref{lemma_F}  to get
\begin{equation}
\label{eq1533}
F(\beta_{n+1})^2 \geq c\,(\beta_{n+1}-\beta_n)^{p+2}.
\end{equation}
Remember that 
\[
\beta_{n}=g_{\frac{p-2}{2}}^{-1}(m_j+k_n) \qquad \mbox{ and }\qquad \beta_{n+1}=g_{\frac{p-2}{2}}^{-1}(m_j+k_{n+1}). 
\]
If we use again that for every \(s, t \in \mathbb{R}\),
\[
\left|g_\frac{p-2}{2}(t) - g_\frac{p-2}{2}(s)\right| \leq C \,\left(|t|^{\frac{p-2}{2}} + |s|^{\frac{p-2}{2}}\right)\,|t-s|,
\]
then one gets
\[
|k_{n+1}-k_n|^{p+2} = \big|(k_{n+1}+m_j)-(k_n+m_j)\big|^{p+2}\leq C \left(|\beta_{n+1}|^{\frac{p-2}{2}} + |\beta_{n}|^{\frac{p-2}{2}} \right)^{p+2} (\beta_{n+1}-\beta_n)^{p+2}.
\]
By using \eqref{betaL} and \eqref{eq1533} we obtain
\[
|k_{n+1}-k_n|^{p+2}\leq C\,L^{\frac{p^2-4}{2}}_R\,F(\beta_{n+1})^2.
\]
so that by \eqref{eq1529},
\[
|A_{n+1}|\,|k_{n+1}-k_n|^{p+2}\leq C\,\frac{4^n\,L^{\frac{p^2-4+2\,p}{2}}_R\,M_{j}^{\frac{4}{p}}}{R}\,|A_n|^{\frac{3}{2}}.
\]
By definition of \(k_n\), the previous inequality gives
\[
\frac{|A_{n+1}|}{R^2}\leq C\, \left(\frac{2^{n\,(p+4)}}{\alpha^{p+2}}\,L^{\frac{p^2-4+2\,p}{2}}_R\,M_{j}^{\frac{4}{p}-p-2}\right)\,\left(\frac{|A_n|}{R^2}\right)^{\frac{3}{2}}.
\]
Since $M_j>0$,  the right-hand side is well-defined.
If we now set \(Y_n=|A_n|/R^2\), this finally yields
\[
Y_{n+1} \leq \left(C_0\,L^{\frac{p^2-4+2\,p}{2}}_R\,M_{j}^{\frac{4}{p}-p-2}\right)\,\left(2^{p+4}\right)^n\, Y_{n}^{\frac{3}{2}}, \qquad \mbox{ for every }n\in\mathbb{N}\setminus\{0\}.
\]
for some $C_0=C_0(\alpha,p)$ which can be supposed to be larger than $1$.
If follows from Lemma \ref{lemma-Santambrogio-Vespri-1} below that
\[
\lim_{n\to +\infty}Y_n = 0,\qquad \mbox{ provided that }\ Y_1 \leq \frac{(2^{p+4})^{-6}}{C_0^2}\, L^{4-p^2-2\,p}_R\,M_{j}^{2\,p+4\,\left(1-\frac{2}{p}\right)},
\]
The condition on $Y_1$ means
\begin{equation}\
\label{nu}
|\{V_j>(1-\alpha)\,M_j\}\cap B_R| \leq \nu\, M_{j}^{2\,p+4\,\left(1-\frac{2}{p}\right)}\,|B_R|,\qquad \mbox{ with } \nu:=\frac{(2^{p+4})^{-6}}{C_0^2\,\pi}\, L^{4-p^2-2p}_R.
\end{equation}
By assuming this condition and recalling the definition of \(Y_n\), we get
\[
V_j \leq \lim_{n\to +\infty} k_n =\left(1-\frac{\alpha}{2}\right)\,M_j,\qquad \mbox{ a.\,e. on } B_{R/2}. 
\]
This completes the proof.
\end{proof}
\begin{oss}[Quality of the constant $\nu$]
\label{oss:demi}
For later reference, it is useful to record that
\[
\nu\,M_{j}^{2\,p+4\,\left(1-\frac{2}{p}\right)}< \frac{1}{2}.
\]
This follows by direct computation, using the definition of $\nu$ and observing that
\[
M_{j}\le 2\,\|v_j\|_{L^\infty(B_R)}=2\,\|U_{x_j}\|^\frac{p}{2}_{L^\infty(B_R)}\le 2\,(L_R-1)^\frac{p}{2}.
\]
Also observe that by its definition \eqref{nu}, the constant $\nu$ is monotone non-increasing as a function of the radius of the ball $B_R$ (since $R\mapsto L_R$ is monotone non-decreasing and $4-p^2-2\,p<0$ for $p\ge 2$).
\end{oss}

\subsection{Alternatives}

\begin{lm}
\label{lm:5SV}
We still use the notation in \eqref{notations} and \eqref{notations2}.
Let \(B_R\Subset \Omega\) and let \(\nu \) be the constant in Lemma \ref{lemmaDeGiorgi}, for \(\alpha=1/4\).
If we set
\[
\delta=\sqrt{\frac{\nu}{2}\,M_{j}^{2\,p+4\,\left(1-\frac{2}{p}\right)}},
\]
then one of the two following alternatives occur:
\begin{itemize}
\item[$(\mathbf{B}_1)$] either 
\begin{equation}
\label{eq890}
\osc_{B_{\delta R}} v_j \leq \frac{7}{8}\, \osc_{B_{R}} v_j,
\end{equation}
\vskip.2cm
\item[$(\mathbf{B}_2)$] or
\begin{equation}
\label{dessous}
\int_{B_{R}\setminus B_{\delta R}}|\nabla v_j|^2 \,dx \geq \frac{1}{512\,\pi}\,\nu\, M_j^2\, M_{j}^{2\,p+4\,\left(1-\frac{2}{p}\right)}. 
\end{equation}
\end{itemize}
\end{lm}
\begin{proof}
We can suppose that $M_j>0$, otherwise there is nothing to prove.
We have two possibilities: either
\[
\left|\left\{V_j>\frac{3}{4}\,M_j\right\}\cap B_R\right| < \nu\,M_{j}^{2\,p+4\,\left(1-\frac{2}{p}\right)}\, |B_R|,
\]
or not. In the first case, by Lemma \ref{lemmaDeGiorgi} with $\alpha=1/4$ we obtain
\[
\osc_{B_{\delta R}} v_j \leq \osc_{B_{R/2}} v_j \leq \frac{7}{8} \osc_{B_{R}} v_j,
\]
which corresponds to alternative $(\mathbf{B}_1)$ in the statement. In the first inequality we used that $\delta<1/2$, see Remark \ref{oss:demi}.
\par
In the second case, 
we appeal to Lemma \ref{lemma-Santambrogio-Vespri-5} with the choices
\[
\varphi=V_j,\qquad M=M_j\qquad \mbox{ and }\qquad \gamma=\nu\,M_j^{2\,p+4\,\left(1-\frac{2}{p}\right)},
\] 
with $\delta$ as in the statement above.
It follows that: 
\begin{itemize}
\item[$\bullet$] either 
\[
\int_{B_R\setminus B_{\delta R}} |\nabla V_j|^2 \,dx \geq \frac{1}{512\pi}\,\nu\, M_j^2\,M_{j}^{2\,p+4\,\left(1-\frac{2}{p}\right)};
\]
\vskip.2cm
\item[$\bullet$] or the following subset of \([\delta\, R , R]\)
\[
\mathcal{A}=\left\{s\in[\delta R,R]\, :\, V_j\geq \frac{5}{8}\,M_j,\ \textrm{ $\mathcal{H}^1-$a.\,e. on } \partial B_s\right\},
\]
has positive measure.
\end{itemize}
If the first possibility occurs, then we are done since this coincides with alternative $(\mathbf{B}_2)$.
\par
In the second case, we consider $u^\varepsilon$ the solution of the regularized problem \eqref{approximation} in a ball $B\Subset \Omega$ such that $B_R\Subset B$. Then we know from Lemma \ref{lm:tracce}
\[
\lim_{k\to +\infty }\left\||u^{\varepsilon_k}_{x_j}|^\frac{p-2}{2}\,u^{\varepsilon_k}_{x_j}-v_j\right\|_{L^\infty(\partial B_s)}=0,\qquad \mbox{ for a.\,e. }s\in[0,R],
\]
for an infinitesimal sequence $\{\varepsilon_k\}_{n\in\mathbb{N}}$.
Since $\mathcal{A}$ has positive measure, we can then choose a radius $s\in\mathcal{A}$ such that the previous convergence holds. For every $n\in\mathbb{N}\setminus\{0\}$, by taking $k$ large enough we thus obtain
\[
|u^{\varepsilon_k}_{x_j}|^\frac{p-2}{2}\,u^{\varepsilon_k}_{x_j}\ge \frac{5}{8}\,M_j+m_j-\frac{1}{n},\ \textrm{ $\mathcal{H}^1-$a.\,e. on } \partial B_s.
\]
We can now apply the minimum principle of Lemma \ref{lm:minimum} with $C=5/8\,M_j+m_j-1/n$ and get
\begin{equation}
\label{quasi}
|u^{\varepsilon_k}_{x_j}|^\frac{p-2}{2}\,u^{\varepsilon_k}_{x_j}\ge \frac{5}{8}\,M_j+m_j-\frac{1}{n},\quad \mbox{ in } B_s.
\end{equation}
Thanks to Proposition \ref{lm_convergence_ueps}, we know that $\{|u^{\varepsilon_k}_{x_j}|^\frac{p-2}{2}\,u^{\varepsilon_k}_{x_j}\}_{k\in\mathbb{N}}$ converges strongly in $L^2(B_s)$ to $v_j$. It then follows from \eqref{quasi} that
\[
v_j\ge \frac{5}{8}\,M_j+m_j-\frac{1}{n},\ \mbox{ a.\,e. in } B_s,\qquad \mbox{ that is } V_j\ge \frac{5}{8}\,M_j-\frac{1}{n},\ \mbox{ a.\,e. in } B_s.
\] 
Hence, by arbitrariness of $n$ we get
\[
\osc_{B_{\delta R}} v_j \leq \osc_{B_{s}} v_j \leq \sup_{B_R} V_j
-\inf_{B_s} V_j\le \frac{3}{8}M_j,
\]
which implies again alternative $(\mathbf{B}_1)$.
The proof is complete.
\end{proof}

\section{Decay estimates for the gradient for $1<p\le 2$}
\label{sec:5}

\subsection{A De Giorgi-type Lemma}

For every $B_R\Subset \Omega$, we introduce the alternative notation
\begin{equation}
\label{notationsbis}
m_j=\inf_{B_R} U_{x_j},\qquad V_j = U_{x_j}-m_j, \qquad M_j=\sup_{B_R} V_j = \osc_{B_R} U_{x_j},\qquad j=1,2,
\end{equation}
and still use the notation \eqref{notations2} for $L_R$.
\begin{lm}
\label{lemmaDeGiorgibis}
Let \(B_R\Subset \Omega\) and \(0<\alpha <1\). By using the notation in \eqref{notationsbis} and \eqref{notations2}, there exists a constant \(\nu=\nu(p,\alpha,L_R)>0\) such that if
\[
\Big|\{V_j>(1-\alpha)\,M_j\} \cap B_R\Big| \leq \nu\, M_{j}^{2}\,|B_R|,
\]
then
\[
0\leq V_j \leq \left(1-\frac{\alpha}{2}\right)\, M_j,\ \mbox{ on } B_{\frac{R}{2}}.
\]
\end{lm}
\begin{proof}
We first observe that if $M_j=0$, then $V_j$ identically vanishes in $B_R$ and there is nothing to prove. Thus, we can assume that $M_j>0$.
\par
For \(n\geq 1\), we set
\[
k_n=M_j\left(1-\frac{\alpha}{2} - \frac{\alpha}{2^n} \right), \qquad R_n=\frac{R}{2} + \frac{R}{2^n}, \qquad A_n=\{V_j>k_n\}\cap B_{R_n},
\]
where the ball $B_{R_n}$ is concentric with $B_R$.
Let \(\theta_n\) be a cut-off function such that
\[
0\le \theta_n\le 1,\qquad\theta_n  \equiv 1 \textrm{ on } B_{R_{n+1}}, \qquad \theta_n \equiv 0 \textrm{ on } \mathbb{R}^2\setminus B_{R_n}
\] 
\[
|\nabla \theta_n| \leq C\,\frac{2^n}{R}\qquad \mbox{ and }\qquad |\nabla \theta_n| \leq C\,\frac{4^n}{R^2}.
\]
We then set for every \(n\geq 1\)
\begin{equation}
\label{betanbis}
\beta_n=m_j+k_n.
\end{equation}
For every $\delta>0$, we take a $C^1$ non-decreasing function $\xi_\delta:\mathbb{R}\to[0+\infty)$ such that\footnote{One can take for example the function $\xi_\delta$ of the form
\[
\xi_\delta(t)=\left\{\begin{array}{ll}
0&, \mbox{for } t\le 0,\\
t^3/\delta^2&, \mbox{for } 0<t<\delta,\\
3\,t-2\,\delta&, \mbox{for } t\ge \delta.
\end{array}
\right.
\]}
\[
\xi_\delta(t)=0,\ \mbox{ for } t\le 0,\qquad |\xi_\delta'(t)|\le C,\qquad \mbox{ for } t\in\mathbb{R},
\]
and
\[
\xi'_\delta(t)=C,\qquad \mbox{ for } t\ge \delta, 
\]
for some universal constant $C>0$. This has to be thought as a smooth approximation of the ``positive part'' function, up to the constant $C>0$. 
In the setting of Proposition \ref{cacciople2}, we take 
\[
\zeta(t)=\xi_\delta(t-\beta_n)\qquad \mbox{ and }\qquad \eta=\theta_n.
\] 
We observe that
\[
\zeta(t)\le C\,(t-\beta_n)_+,
\]
so that
\begin{equation}
\label{piero}
\zeta(U_{x_j})\le C\,(U_{x_j}-m_j-k_n)_+\le C\, M_j\le 2\,C\, L_R.
\end{equation}
By using \eqref{piero}, the definition of \(A_n\) and the properties of $\zeta$, one gets from \eqref{uniformesobsub2limit}
\[
\begin{split}
C\,\sum_{i=1}^{2}\int_{\{U_{x_j}\ge \beta_n+\delta\}\cap \{U_{x_i}\not=0\}} \left|U_{x_i}\right|^{p-2}\,|U_{x_j\,x_i}|^2\,\theta_n^2\,dx
& \le C\,\int_{\{U_{x_j}\ge \beta_n\}} |\nabla U|^p\,\Big(|\nabla \theta_n|^2+|D^2\theta_n|\Big)\,dx\\
&+\int_{\{U_{x_j}\ge \beta_n\}} |\nabla U|^{p-1}\,|\zeta(U_{x_j})|\,\Big(|\nabla \theta_n|^2+|D^2\theta_n|\Big)\,dx\\ 
&\le C\,L_R^{p}\,\int_{\{U_{x_j}\ge \beta_n\}} \Big(|\nabla \theta_n|^2+|D^2\theta_n|\Big)\,dx.
\end{split}
\]
Since \(p<2\) and \(|U_{x_i}|\leq L_R\) a.e., one gets
\[
\begin{split}
\sum_{i=1}^{2}\int_{\{U_{x_j}\ge \beta_n+\delta\}} |U_{x_j\,x_i}|^2\,\theta_n^2\,dx
&\le C\,L_R^2\,\int_{\{U_{x_j}\ge \beta_n\}}\Big(|\nabla \theta_n|^2+|D^2\theta_n|\Big)\,dx.
\end{split}
\]
Here, we have also used the fact that \(U_{x_j\,x_i}=0\) a.e. on the set \(\{U_{x_i}=0\}\).
We now take the limit as $\delta$ goes to $0$ in the left-hand side. By the Monotone Convergence Theorem, we get
\[
\sum_{i=1}^{2}\int_{\{U_{x_j}\ge \beta_n\}} |U_{x_j\,x_i}|^2\,\theta_n^2\,dx
\le C\,L_R^2\,\int_{\{U_{x_j}\ge \beta_n\}}\Big(|\nabla \theta_n|^2+|D^2\theta_n|\Big)\,dx.
\]
In view of the properties of \(\theta_n\), it follows that
\begin{equation}
\label{caccio3}
\int \left|\nabla\left(U_{x_j}-\beta_n\right)_+\right|^2\,\theta_n^2\,dx\leq  C\,L^2_R\, 4^n\, \frac{|A_n|}{R^2},
\end{equation}
for some $C=C(p)>0$. Observe that
\begin{equation}
\label{caccio4}
\int |\nabla \theta_n|^2\,(U_{x_j}-\beta_n)_+^2\,dx \leq C\,L^{2}_R\,4^n\,\frac{|A_n|}{R^2},
\end{equation}
thanks to \eqref{piero}. By adding \eqref{caccio3} and \eqref{caccio4}, we get 
\[
\int_{B_{R_n}}|\nabla \left((U_{x_j}-\beta_n)_+\,\theta_n\right)|^2\,dx \leq C\,4^n\, L^2_R\,\frac{|A_n|}{R^2},
\]
where as usual $C=C(p)>0$.
We rely again on the Poincar\'e inequality and obtain
\[
\begin{split}
\left|\{x\in B_{R_n}\,:\, (U_{x_j}-\beta_n)_+\,\theta_n>0\}\right|& \int_{B_{R_n}}|\nabla \left((U_{x_j}-\beta_n)_+\,\theta_n\right)|^2\,dx\\
& \geq c\, \int_{B_{R_n}}|(U_{x_j}-\beta_n)_+\,\theta_n|^2\,dx.
\end{split}
\]
Since \(\theta_n\equiv 1\) on \(B_{R_{n+1}}\) and by construction
\[
|A_n| \geq \left|\{(U_{x_j}-\beta_n)_+\,\theta_n>0\}\right|,
\]
one gets
\[
\int_{B_{R_{n+1}}} (U_{x_j}-\beta_n)_+^2\,dx \leq C\,\frac{4^n\, L^{2}_R}{R^2}\, |A_n|^2,
\]
for some $C=C(p)>0$.
By using that 
\[
A_{n+1}=\{V_j>k_{n+1}\}\cap B_{R_{n+1}}=\{U_{x_j}>\beta_{n+1}\}\cap B_{R_{n+1}},
\]
we obtain
\[
\int_{B_{R_{n+1}}}(U_{x_j}-\beta_n)_+^2 \,dx\geq \int_{A_{n+1}}\,(U_{x_j}-\beta_n)_+^2\,dx \geq |A_{n+1}|\,(\beta_{n+1}-\beta_n)^2.
\]
This gives
\begin{equation}\label{eq15292}
|A_{n+1}|\,(\beta_{n+1}-\beta_n)^2 \leq C\, \frac{4^n\,L^2_R}{R^2}\, |A_n|^2.
\end{equation}
By recalling the definition of \(\beta_n\) and $k_n$, the previous inequality gives
\[
\frac{|A_{n+1}|}{R^2}\leq C\, \left(\frac{4^{2\,n}}{\alpha^2}\,L^2_R\,M_{j}^{-2}\right)\,\left(\frac{|A_n|}{R^2}\right)^2.
\]
Since $M_j>0$,  the right-hand side is well-defined.
As before, we set \(Y_n=|A_n|/R^2\) and obtain
\[
Y_{n+1} \leq \left(C_0\,L^2_R\,M_{j}^{-2}\right)\,16^n\, Y_{n}^2, \qquad \mbox{ for every }n\in\mathbb{N}\setminus\{0\},
\]
for some $C_0=C_0(\alpha,p)\ge 1$.
Again by Lemma \ref{lemma-Santambrogio-Vespri-1} we get
\[
\lim_{n\to +\infty}Y_n = 0,\qquad \mbox{ provided that }\ Y_1 \leq \frac{(16)^{-2}}{C_0}\, L^{-2}_R\,M_{j}^2,
\]
This means 
\[
|\{V_j>(1-\alpha)\,M_j\}\cap B_R| \leq \nu\, M_{j}^2\,|B_R|,\qquad \mbox{ with } \nu:=\frac{16^{-2}}{C_0^2\,\pi}\, L^{-2}_R.
\]
By assuming this condition and recalling the definition of \(Y_n\), we get
\[
V_j \leq \lim_{n\to +\infty} k_n =\left(1-\frac{\alpha}{2}\right)\,M_j,\qquad \mbox{ a.\,e. on } B_{R/2}. 
\]
This completes the proof.
\end{proof}

\begin{oss}[Quality of the constant $\nu$]
\label{oss:demibis}
For later reference, as in the previous case we observe that
\[
\nu\,M_{j}^2< \frac{1}{2},
\]
and that the constant $\nu$ is monotone non-increasing as a function of $R$.
\end{oss}

\subsection{Alternatives}

\begin{lm}
\label{lm:5SVbis}
We still use the notation in \eqref{notationsbis} and \eqref{notations2}.
Let \(B_R\Subset B_{2\,R}\Subset \Omega\) and let \(\nu \) be the constant in Lemma \ref{lemmaDeGiorgibis}, for \(\alpha=1/4\).
If we set
\[
\delta=\sqrt{\frac{\nu}{2}\,M_{j}^2},
\]
then one of the two following alternatives occur:
\begin{itemize}
\item[$(\mathbf{B}_1)$] either 
\begin{equation}
\label{eq890bis}
\osc_{B_{\delta R}} U_{x_j} \leq \frac{7}{8}\, \osc_{B_{R}} U_{x_j},
\end{equation}
\vskip.2cm
\item[$(\mathbf{B}_2)$] or
\begin{equation}
\label{dessousbis}
\int_{B_{R}\setminus B_{\delta R}}|\nabla U_{x_j}|^2 \,dx \geq \frac{1}{512\pi}\,\nu\, M_j^4. 
\end{equation}
\end{itemize}
\end{lm}
\begin{proof}
We can suppose that $M_j>0$, otherwise there is nothing to prove.
We have two possibilities: either
\[
\left|\left\{V_j>\frac{3}{4}\,M_j\right\}\cap B_R\right| < \nu\,M_{j}^2\, |B_R|,
\]
or not. In the first case, by Lemma \ref{lemmaDeGiorgibis} with $\alpha=1/4$ we obtain
\[
\osc_{B_{\delta R}} U_{x_j} \leq \osc_{B_{R/2}} U_{x_j} \leq \frac{7}{8} \osc_{B_{R}} U_{x_j},
\]
which corresponds to alternative $(\mathbf{B}_1)$ in the statement. In the first inequality we used again that $\delta<1/2$, see Remark \ref{oss:demibis}.
\par
In the second case, 
we appeal to Lemma \ref{lemma-Santambrogio-Vespri-5} with the choices
\[
\varphi=V_j,\qquad M=M_j\qquad \mbox{ and }\qquad \gamma=\nu\,M_j^2,
\] 
with $\delta$ as in the statement above.
It follows that: 
\begin{itemize}
\item[$\bullet$] either 
\[
\int_{B_R\setminus B_{\delta R}} |\nabla V_j|^2 \,dx \geq c\,\nu\, M_j^4,
\]
for some universal constant \(c>0\);
\vskip.2cm
\item[$\bullet$] or the set
\[
\mathcal{A}=\left\{s\in[\delta R,R]\, :\, U_{x_j}-m_j\geq \frac{5}{8}\,M_j,\ \textrm{ $\mathcal{H}^1-$a.\,e. on } \partial B_s\right\},
\]
has positive measure.
\end{itemize}
Again, if the first possibility occurs, then we are done since this coincides with alternative $(\mathbf{B}_2)$.
\par
In the second case, we consider $u^\varepsilon$ the solution of the regularized problem \eqref{approximationsub} in a ball $B\Subset \Omega$ such that $B_R\Subset B$. Then we know from Lemma \ref{lm:traccebis} 
\[
\lim_{k\to +\infty}\|u^{\varepsilon_k}_{x_j}-U_{x_j}\|_{L^\infty(\partial B_s)}=0,\qquad \mbox{ for a.\,e. }s\in[0,R].
\]
for an infinitesimal sequence $\{\varepsilon_k\}_{k\in\mathbb{N}}$.
Since $\mathcal{A}$ has positive measure, we can then choose a radius $s\in\mathcal{A}$ such that the previous convergence holds. For every $n\in\mathbb{N}\setminus\{0\}$, by taking $k$ large enough we thus obtain
\[
u^\varepsilon_{x_j}\ge \frac{5}{8}\,M_j+m_j-\frac{1}{n},\ \textrm{ $\mathcal{H}^1-$a.\,e. on } \partial B_s,
\]
By proceeding as in the proof of Lemma \ref{lm:5SV} and using this time the minimum principle of Lemma \ref{lm:minimumbis} and Proposition \ref{prop:convergence_uep1p2bis}, we obtain
\[
 U_{x_j}-m_j\ge \frac{5}{8}\,M_j-\frac{1}{n},\ \mbox{ a.\,e. in } B_s.
\] 
By arbitrariness of $n$, we get
\[
\osc_{B_{\delta R}} U_{x_j} \leq \osc_{B_{s}} U_{x_j} \leq \left(\sup_{B_R} U_{x_j}-m_j\right)
-\left(\inf_{B_s} U_{x_j}-m_j\right)\le \frac{3}{8}M_j,
\]
which is again alternative $(\mathbf{B}_1)$.
The proof is complete.
\end{proof}

\section{Proof of the Main Theorem}
\label{sec:6}

\subsection{Case $p>2$}
We already observed that for every $q>-1$ the function $t\mapsto t\,|t|^q$ is a homeomorphism on \(\mathbb{R}\). This implies the following
\begin{lm}
\label{lemma-continuity-ui2vi}
Let $f:E\to \mathbb{R}$ be a measurable function such that for some \(q>-1\) the function \(|f|^q\,f\) is continuous. Then \(f\) itself is continuous. 
\end{lm}
\noindent
In view of this result, in order to the Main Theorem in the case $p>2$ it is sufficient to prove that each function
\[
v_j=|U_{x_j}|^{\frac{p-2}{2}}\,U_{x_j},\qquad j=1,2,
\]
is continuous on $\Omega$. Thus Theorem \ref{teo:theorem-continuity-N2} for $p>2$ is a consequence of the following
\begin{prop} 
Let $p>2$, $x_0\in\Omega$ and \(R_0>0\) such that $B_{R_0}(x_0)\Subset \Omega$. We consider the family of balls $\{B_R(x_0)\}_{0<R\le R_0}$ centered at $x_0$.
Then we have
\[
\lim_{R\searrow 0} \left(\osc_{B_R(x_0)} v_j\right) =0,\qquad j=1,2.
\]
\end{prop}
\begin{proof} 
For simplicity, in what follows we omit to indicate the center $x_0$ of the balls.
Since the map \(R\mapsto \osc_{B_R} v_j\) is non decreasing, we only need to find a decreasing sequence \(\{R_n\}_{n\in\mathbb{N}}\) converging to \(0\) such that 
\[
\lim_{n\to +\infty} \left(\osc_{B_{R_n}} v_j\right)=0.
\]
For simplicity we now drop the index \(j\) and write \(v\) in place of \(v_j\). We set 
\[
M_0=\osc_{B_{R_0}} v\qquad \mbox{ and }\qquad \delta_0=\sqrt{\frac{\nu_0}{2}\,M_0^{2\,p+4\,\left(1-\frac{2}{p}\right)}},
\] 
where $\nu_0$ is the constant of Lemma \ref{lemmaDeGiorgi} for $R=R_0$ and \(\alpha=1/4\).
We construct by induction the sequence of triples \(\{R_n, M_n, \delta_n\}_{n\in \mathbb{N}}\) defined by
\[
M_n:=\osc_{B_{R_n}} v,\qquad \delta_n = \sqrt{\frac{\nu_n}{2}\,M_n^{2\,p+4\,\left(1-\frac{2}{p}\right)}}, \quad R_{n+1}=\delta_n\, R_n,
\]
and $\nu_n$ is the constant of Lemma \ref{lemmaDeGiorgi} for $R=R_n$ and \(\alpha=1/4\).
Since \(\delta_n<1/2\) for every \(n\in \mathbb{N}\) (see Remark \ref{oss:demi}), the sequence \(\{R_n\}_{n\in \mathbb{N}}\) is monotone decreasing and goes to \(0\). In order to conclude, we just need to prove that
\begin{equation}
\label{bof}
\lim_{n\to\infty} M_n=0.
\end{equation}
Observe that we can suppose $M_n>0$ for every $n\in\mathbb{N}$, otherwise there is nothing to prove.
We set
\[
I:=\left\{ n\in \mathbb{N}\, :\, \int_{B_{R_n}\setminus B_{R_{n+1}}} |\nabla v|^2\,dx \geq \frac{1}{512\,\pi}\,\nu_n\,M_{n}^{2\,p+4\,\left(1-\frac{4}{p}\right)}\,M_n^2\right\},
\]
and we have
\begin{equation}
\label{serie}
\begin{split}
\frac{\nu_0}{512\,\pi}\,\sum_{n\in I}\,M_{n}^{2\,p+2+4\,\left(1-\frac{2}{p}\right)} &\le \frac{1}{512\,\pi}\,\sum_{n\in I}\nu_n\,M_{n}^{2\,p+2+4\,\left(1-\frac{2}{p}\right)}\\
& \leq \sum_{n\in I} \int_{B_{R_n}\setminus B_{R_{n+1}}}|\nabla v|^2\,dx\leq \int_{B_{R_0}} |\nabla v|^2\,dx,
\end{split}
\end{equation}
thanks to the fact that $\nu_n\ge \nu_0>0$ for every $n\in\mathbb{N}$ (see Remark \ref{oss:demi}).
We now have two possibilies: either $I$ is infinite or it is finite.
If the first alternative occurs, then \eqref{serie} and the fact that $v\in W^{1,2}_{\rm loc}(\Omega)$ imply
\[
\lim_{I\ni n\to \infty} M_n = 0.
\]
This means that the monotone sequence \(\{M_n\}_{n\in \mathbb{N}}\) has a subsequence which converges to \(0\), thus we have \eqref{bof}  and this completes the proof in that case.
\par
Otherwise, if \(I\) is finite then there exists \(\ell\in \mathbb{N}\) such that for every \(n\geq \ell\) we have
\[
\int_{B_{R_n}\setminus B_{R_{n+1}}} |\nabla v|^2\,dx< \frac{1}{512\pi}\,\nu_n\,M_{n}^{2\,p+4\,\left(1-\frac{2}{p}\right)}\,M_n^2.
\]
By Lemma \ref{lm:5SV}, this in turn implies that 
\[
M_{n+1}=\osc_{B_{R_{n+1}}} v\leq \frac{7}{8}\,\osc_{B_{R_n}} v=\frac{7}{8}\, M_n,\qquad \mbox{ for every } n\ge \ell.
\] 
This again implies \eqref{bof}. The proof is complete.
\end{proof}
\subsection{Case $1<p\le 2$}
The case $1<p\le 2$ is similar, but more direct. This time Theorem \ref{teo:theorem-continuity-N2} follows from the following result, whose proof is exactly as above. It is sufficient to use Lemma \ref{lemmaDeGiorgibis} in place of Lemma \ref{lemmaDeGiorgi} and Lemma \ref{lm:5SVbis} in place of Lemma \ref{lm:5SV}. We leave the details to the reader.
\begin{prop} 
Let $1<p\le 2$, $x_0\in\Omega$ and \(R_0>0\) such that $B_{R_0}(x_0)\Subset \Omega$. We consider the family of balls $\{B_R(x_0)\}_{0<R\le R_0}$ centered at $x_0$.
Then we have
\[
\lim_{R\searrow 0} \left(\osc_{B_R(x_0)} U_{x_j}\right) =0,\qquad j=1,2.
\]
\end{prop}

\appendix

\section{Inequalities}

In the proof of Lemma \ref{lemmaDeGiorgibis} we crucially relied on the following double-sided estimate for the function
\[
F(t)=\frac{p}{2}\,\int_{\beta}^{t}|s|^{\frac{p-2}{2}}(s-\beta)_+\,ds, \quad t\in \mathbb{R}.
\]
\begin{lm}
\label{lemma_F} Let $\beta\in\mathbb{R}$ and $p> 2$.
There exist a constant \(C=C(p)>1\) such that for every \(t\in \mathbb{R}\),
\begin{equation}
\label{eq1453}
\frac{1}{C}\,(t-\beta)^{\frac{p+2}{2}}_+\leq F(t) \leq C\, \left(|t|^\frac{p-2}{2}+\left(\max\{0,-\beta\}\right)^\frac{p-2}{2}\right)\,(t-\beta)^2_+.
\end{equation}
\end{lm}
\begin{proof}
Since \(F(t)=0\) when \(t\leq \beta\), both inequalities are true in this case. Thus let us assume that \(t>\beta\). Moreover, if \(\beta=0\),
\[
F(t)=\frac{p}{2}\,\int_{0}^{t}s^{\frac{p-2}{2}}\,s\,ds = \frac{p}{p+2}\,t^{\frac{p+2}{2}},\qquad \mbox{ for }t>0,
\]
which implies the result. 
\vskip.2cm\noindent
{\bf Case \(\beta>0\)}. By H\"older's inequality
\[
\begin{split}
\frac{(t-\beta)^{p}_+}{2^\frac{p}{2}}=\left(\int_{\beta}^t (s-\beta)_+\,ds\right)^\frac{p}{2}&=\left(\int_{\beta}^t s^{\frac{p-2}{p}}\,\frac{(s-\beta)_+}{s^{\frac{p-2}{p}}}\,ds\right)^\frac{p}{2}\\
&\le \left(\int_{\beta}^t s^\frac{p-2}{2}\,(s-\beta)_+\,ds\right)\,\left(\int_{\beta}^t \frac{(s-\beta)_+}{s}\,ds\right)^\frac{p-2}{2}\\
&\le \frac{2}{p}\,F(t)\,(t-\beta)^\frac{p-2}{2}_+,
\end{split}
\]
where we used that $(s-\beta)_+\le s$ and this gives the lower bound in \eqref{eq1453}. As for the upper bound, by the change of variables \(\tau=s/\beta\) one has
\[
F(t)= \beta^{\frac{p+2}{2}}\, F_+\left(\frac{t}{\beta}\right),\qquad \mbox{ where }\ F_+(X) =\frac{p}{2}\, \int_{1}^X \tau^{\frac{p-2}{2}}\,(\tau-1) \,d\tau, \qquad \tau>1.
\]
Observe that
\[
F_+(X)= \frac{p}{p+2}\,\left(X^{\frac{p+2}{2}}-1\right) -\left(X^\frac{p}{2}-1\right),\qquad X>1.
\]
Moreover, by convexity of the function $X\mapsto X^{p/2}$ we have
\[
-\left(X^\frac{p}{2}-1\right)\le -\frac{p}{2}\,(X-1),
\]
while a second order Taylor expansion gives
\[
\frac{p}{p+2}\,\left(X^{\frac{p+2}{2}}-1\right)=\frac{p}{2}\,(X-1)+\frac{p^2}{4}\,\int_1^X s^\frac{p-2}{2}\,(X-s)\,ds\le \frac{p}{2}(X-1)+\frac{p^2}{8}\,X^\frac{p-2}{2}\,(X-1)^2.
\]
Thus we obtain
\[
F_+(X)\le \frac{p^2}{8}\,X^\frac{p-2}{2}\,(X-1)^2,\qquad X>1,
\]
and finally for $t>\beta$
\[
F(t)=\beta^{\frac{p+2}{2}}\, F_+\left(\frac{t}{\beta}\right)\le \frac{p^2}{8}\,t^\frac{p-2}{2}\,(t-\beta)^2,
\]
which proves the upper bound in \eqref{eq1453}.
\vskip.2cm\noindent
{\bf Case $\beta<0$}.
This case is slightly more complicated. We introduce the function
\[
F_-(X) =\frac{p}{2}\, \int_{-1}^X |s|^{\frac{p-2}{2}}\,(s+1) \,ds = \frac{p}{p+2}\,\left(|X|^{\frac{p+2}{2}}-1\right) +\left(|X|^\frac{p-2}{2}\,X+1\right), \qquad X>-1.
\]
It is sufficient to prove that there exists \(C>1\) such that
\begin{equation}
\label{modified}
\frac{1}{C}\,(X+1)^{\frac{p+2}{2}} \leq F_-(X) \leq C\,\left(|X|^\frac{p-2}{2}+1\right)\,(X+1)^2.
\end{equation}
Indeed, \(F(t) =|\beta|^{(p+2)/2}\, F_-(t/|\beta|)\) and this would give
\[
\frac{1}{C}\,(t-\beta)^{\frac{p+2}{2}}\leq F(t) \leq C\, \left(|t|^\frac{p-2}{2} +|\beta|^\frac{p-2}{2}\right)\,(t-\beta)^2,
\]
as desired. 
\par
The upper bound in \eqref{modified} for $-1<X<0$ can be obtained as before, by using a second order Taylor expansion for the first term and using that $\tau\mapsto |\tau|^{(p-2)/2}\,\tau$ is concave on $-1<\tau<0$. This gives
\[
\begin{split}
F_-(X)&=\frac{p}{p+2}\,\left(|X|^{\frac{p+2}{2}}-1\right) +\left(|X|^\frac{p-2}{2}\,X+1\right)\\
&\le -\frac{p}{2}\,(X+1)+\frac{p^2}{4}\,\int_{-1}^X |s|^\frac{p-2}{2}\,(X-s)\,ds+\frac{p}{2}\,(X+1)\\
&\le \frac{p^2}{8}\,(X+1)^2.
\end{split}
\]
Observe that the upper bound is trivial for $0\le X\le 1$, since
\[
\frac{p}{p+2}\,\left(|X|^{\frac{p+2}{2}}-1\right) +\left(|X|^\frac{p-2}{2}\,X+1\right)\le 2\le 2\,\left(|X|^\frac{p-2}{2}+1\right)\,(X+1)^2.
\]
Finally, for $X>1$ we still use a second order Taylor expansion for the first term and the elementary inequality
\[
X^\frac{p}{2}+1\le \frac{1}{2}\,X^\frac{p-2}{2}\,(X+1)^2,
\]
for the second one. These yield
\[
\begin{split}
F_-(X)&\le \frac{p^2}{4}\,\int_{-1}^X |s|^\frac{p-2}{2}\,(X-s)\,ds+\frac{1}{2}\,X^\frac{p-2}{2}(X+1)^2\le \left(\frac{p^2}{8}+\frac{1}{2}\right)\,X^\frac{p-2}{2}\,(X+1)^2.
\end{split}
\]
In order to prove the lower bound, we just observe that the function
\[
X\mapsto \frac{(X+1)^\frac{p+2}{2}}{F_-(X)},\qquad X>-1
\]
is positive continuous on $(-1,+\infty)$ and such that
\[
\lim_{X\to (-1)^+} \frac{(X+1)^\frac{p+2}{2}}{F_-(X)}<+\infty \qquad \mbox{ and }\qquad \lim_{X\to +\infty} \frac{(X+1)^\frac{p+2}{2}}{F_-(X)}<+\infty.
\]
Thus it is bounded on $(-1,+\infty)$ and this concludes the proof of the lower bound.
\end{proof}

\begin{lm}
Let $1<q\le 2$, for every $z_0,z_1\in\mathbb{R}^N$ we have
\begin{equation}
\label{dibene}
\Big||z_0|^{q-2}\,z_0-|z_1|^{q-2}\,z_1\Big|\le 2^{2-q}\,|z_0-z_1|^{q-1}.
\end{equation}
\end{lm}
\begin{proof}
The proof is the same as in \cite[Lemma 4.4]{DiB}. We first observe that if $z_1=z_0$ there is nothing to prove, thus we can suppose $|z_1-z_0|>0$.
Let us set 
\[
z_t=(1-t) z_0+t\,z_1,\qquad t\in[0,1],
\]
then we have
\[
|z_0|^{q-2}\,z_0-|z_1|^{q-2}\,z_1=\int_0^1 \frac{d}{dt} \left(|z_t|^{q-2}\,z_t\right)\,dt=(q-1)\,\int_0^1 |z_t|^{q-2}\,(z_1-z_0)\,dt,
\]
which implies
\begin{equation}
\label{dibenemedio}
\Big||z_0|^{q-2}\,z_0-|z_1|^{q-2}\,z_1\Big|\le (q-1)\,|z_1-z_0|\,\int_0^1 \Big||z_0|-t\,|z_1-z_0|\Big|^{q-2}\,dt.
\end{equation}
We now distinguish two cases:
\[
\mbox{ either }\quad |z_0|\ge |z_1-z_0|\qquad \mbox{ or }\qquad |z_0|< |z_1-z_0|.
\]
In the first case, we have
\[
\begin{split}
\int_0^1 \Big||z_0|-t\,|z_1-z_0|\Big|^{q-2}\,dt=\int_0^1 \Big(|z_0|-t\,|z_1-z_0|\Big)^{q-2}\,dt&=\frac{|z_0|^{q-1}-\Big(|z_0|-|z_1-z_0|\Big)^{q-1}}{(q-1)\,|z_1-z_0|}\\
&\le \frac{|z_1-z_0|^{q-2}}{q-1},
\end{split}
\]
which inserted in \eqref{dibenemedio} gives the desired conclusion. In the second case, let $0<\kappa<1$ be such that
\[
|z_0|=\kappa\,|z_0-z_1|,
\]
then we have
\[
\begin{split}
\int_0^1 \Big||z_0|-t\,|z_1-z_0|\Big|^{q-2}\,dt&=\int_0^\kappa \Big(|z_0|-t\,|z_1-z_0|\Big)^{q-2}\,dt+\int_\kappa^1 \Big(t\,|z_1-z_0|-|z_0|\Big)^{q-2}\,dt\\
&=\frac{|z_0|^{q-1}}{(q-1)\,|z_1-z_0|}+\frac{\Big(|z_1-z_0|-|z_0|\Big)^{q-1}}{(q-1)\,|z_1-z_0|}\\
&\le 2^{2-q}\,\frac{|z_1-z_0|^{q-2}}{q-1}.
\end{split}
\]
In view of \eqref{dibenemedio}, this gives the desired conclusion.
\end{proof}

\begin{coro}
\label{coro:dibene}
Let $1<p\le 2$, for every $\varepsilon\ge 0$ and every $t,s\in\mathbb{R}$ we have
\[
\Big|(\varepsilon+t^2)^{\frac{p-2}{4}}\,t-(\varepsilon+s^2)^{\frac{p-2}{4}}\,s\Big|\le 2^\frac{2-p}{2}\,|t-s|^\frac{p}{2},\qquad t,s\in\mathbb{R}.
\]
\end{coro}
\begin{proof}
We use \eqref{dibene} with the choices
\[
N=2,\qquad q=\frac{p+2}{2},\qquad z_0=\left(t,\sqrt{\varepsilon}\right)\qquad \mbox{ and }\qquad z_1=\left(s,\sqrt{\varepsilon}\right).
\]
This implies
\[
\Big|(\varepsilon+t^2)^{\frac{p-2}{4}}\,\left(t,\sqrt{\varepsilon}\right)-(\varepsilon+s^2)^{\frac{p-2}{4}}\,\left(s,\sqrt{\varepsilon}\right)\Big|\le 2^\frac{2-p}{2}\,|t-s|^\frac{p}{2}.
\]
By further observing that 
\[
\Big|(\varepsilon+t^2)^{\frac{p-2}{4}}\,\left(t,\sqrt{\varepsilon}\right)-(\varepsilon+s^2)^{\frac{p-2}{4}}\,\left(s,\sqrt{\varepsilon}\right)\Big|\ge \Big|(\varepsilon+t^2)^{\frac{p-2}{4}}\,t-(\varepsilon+s^2)^{\frac{p-2}{4}}\,s\Big|,
\]
we get the conclusion.
\end{proof}
\section{Some general tools}

In the proof of Lemmas \ref{lemmaDeGiorgi} and \ref{lemmaDeGiorgibis}, we used the following classical result. This can be found for example in \cite[Lemma 7.1]{Gi}.
\begin{lm}
\label{lemma-Santambrogio-Vespri-1}
If \(\{Y_n\}_{n\in\mathbb{N}}\) is a sequence of nonnegative numbers satisfying
\[
Y_{n+1}\leq c\, b^n\, Y_{n}^{1+\beta}, \quad  Y_1\leq c^{-1/\beta}b^{-(\beta+1)/\beta^2}, \quad \mbox{ for some }c,\, b,\,\beta>0,
\]
then \(\lim_{n\to+\infty} Y_n=0\).
\end{lm}

The next lemma is a Fubini-type result on the convergence of Sobolev functions.
We denote by $\mathcal{H}^1$ the one-dimensional Hausdorff measure.
\begin{lm}
\label{lm:dellafrancesca}
Let $0<\tau<1$ and $1\le p<\infty$. Let $B_R(x_0)\subset\mathbb{R}^2$ be the disc centered at $x_0$ with radius $R>0$ and let $\{u_n\}_{n\in\mathbb{N}}\subset W^{\tau,p}(B_R(x_0))$ be a sequence strongly converging to $0$, i.e. such that
\[
\lim_{n\to\infty}\left[\int_{B_R(x_0)} |u_n|^p\,dx+\int_{B_R(x_0)}\int_{B_R(x_0)} \frac{|u_n(x)-u_n(y)|^p}{|x-y|^{2+\tau\,p}}\,dx\,dy\right]=0.
\]
Then there exists a subsequence \(\{u_{n_i}\}_{i\in\mathbb{N}}\) such that for almost every $r\in[0,R]$, $\{u_{n_i}\}_{i\in\mathbb{N}}$ strongly converges to $0$ in $W^{\tau,p}(\partial B_r(x_0))$, i.e.
\[
\lim_{i\to\infty}\left[\int_{\partial B_r(x_0)} |u_{n_i}|^p\,d\mathcal{H}^1+\int_{\partial B_r(x_0)}\int_{\partial B_r(x_0)} \frac{|u_{n_i}(x)-u_{n_i}(y)|^p}{|x-y|^{1+\tau\,p}}\,d\mathcal{H}^{1}(x)\,d\mathcal{H}^{1}(y)\right]=0.
\]
\end{lm}
\begin{proof}
Let us consider the convergence of the double integral, that for the $L^p$ norm being similar and simpler. Without loss of generality, we can assume $x_0=0$, then we omit to precise the center of the ball. 
We use polar coordinates $x=\varrho\,e^{i\,\vartheta}$. We need to show that up to a subsequence,
\begin{equation}
\label{chaud}
\lim_{n\to\infty}[u_n]^p_{W^{\tau,p}(\partial B_r)}=\int_{[0,2\,\pi]\times[0,2\,\pi]} \frac{|u_n(\varrho\,e^{i\,\vartheta})-u_n(\varrho\,e^{i\,\omega})|^p}{|e^{i\,\vartheta}-e^{i\,\omega}|^{1+\tau\,p}}\,d\vartheta\,d\omega=0,\quad \mbox{ for a.\,e. }\varrho\in[0,R].
\end{equation}
For every $u\in W^{\tau,p}(\mathbb{R}^2)$ and $\varepsilon>0$, we introduce 
\[
\mathcal{W}_\varepsilon(u):=\int_\varepsilon^\infty \int_{[0,2\,\pi]\times[0,2\,\pi]} \frac{|u(\varrho\,e^{i\,\vartheta})-u(\varrho\,e^{i\,\omega})|^p}{|e^{i\,\vartheta}-e^{i\,\omega}|^{1+\tau\,p}}\,d\vartheta\,d\omega\,\frac{\varrho\,d\varrho}{\varrho^{1+\tau\,p}}.
\]
We claim that
\begin{equation}
\label{Ar}
\mathcal{W}_\varepsilon(u)\le \frac{C}{\varepsilon}\,[u]_{W^{\tau,p}(\mathbb{R}^2)}^p=\frac{C}{\varepsilon}\,\int_{\mathbb{R}^2}\int_{\mathbb{R}^2} \frac{|u(x)-u(y)|^p}{|x-y|^{2+\tau\,p}}\,dx\,dy,
\end{equation}
for some constant $C=C(p,\tau)>0$. Let us assume \eqref{Ar} for a moment and explain how to conclude: we can extend $\{u_n\}_{n\in\mathbb{N}}$ to a sequence $\{\widetilde u_n\}_{n\in\mathbb{N}}\subset W^{\tau,p}(\mathbb{R}^2)$ such that
\[
\widetilde u_n=u_n,\ \mbox{ on } B_R\qquad \mbox{ and }\qquad [\widetilde u_n]_{W^{\tau,p}(\mathbb{R}^2)}^p\le C\,[u_n]_{W^{\tau,p}(B_R)}^p,
\]
see \cite[Lemma 7.45]{Ad}.
The latter and \eqref{Ar} imply that
\[
\lim_{n\to\infty} \mathcal{W}_\varepsilon(\widetilde u_n)=0,\qquad \mbox{ for every }\varepsilon>0.
\]
By definition of $\mathcal{W}_\varepsilon$, this means that the sequence of functions 
\[
f_n(\varrho)=\frac{\varrho}{\varrho^{1+\tau\,p}}\int_{[0,2\,\pi]\times[0,2\,\pi]} \frac{|u_n(\varrho\,e^{i\,\vartheta})-u_n(\varrho\,e^{i\,\omega})|^p}{|e^{i\,\vartheta}-e^{i\,\omega}|^{1+\tau\,p}}\,d\vartheta\,d\omega,
\]
converges to \(0\) in \(L^{p}((\varepsilon, R))\). Hence, there exists a subsequence \(\{f_{n_i}\}_{i\in \mathbb{N}}\) which converges almost everywhere to \(0\) on \((\varepsilon,R)\). By taking a sequence  \(\{\varepsilon_k\}_{k\in \mathbb{N}}\) converging to \(0\) and repeating the above argument for each \(\varepsilon_k\), a diagonal argument leads to the existence of a subsequence still denoted by \(\{f_{n_i}\}_{i\in \mathbb{N}}\) which converges almost everywhere to \(0\) on \((0,R)\). Equivalently,  \(\{u_{n_i}\}_{i\in \mathbb{N}}\) satisfies \eqref{chaud} for almost every $\varrho\in[0,R]$. 
\vskip.2cm\noindent
Let us now show \eqref{Ar}. The proof is similar to that of \cite[Lemma A.4]{BD}.
For $\varrho\ge \varepsilon$, $t\ge 0$ and $\vartheta,\omega\in[0,2\,\pi]$ we have
\[
\begin{split}
\left|u(\varrho\,e^{i\,\vartheta})-u(\varrho\,e^{i\,\omega})\right|^p\le  C\,\left|u(\varrho\,e^{i\,\vartheta})-u\left((\varrho+t)\,e^{i\,\frac{\omega+\vartheta}{2}}\right)\right|^p+C\,\left|u\left((\varrho+t)\,e^{i\,\frac{\omega+\vartheta}{2}}\right)-u(\varrho\,e^{i\,\omega})\right|^p,
\end{split}
\]
and (for $\vartheta\not=\omega$)
\[
\varrho^{-\tau\,p-1}\,|e^{i\,\vartheta}-e^{i\,\omega}|^{-\tau\,p-1}=(1+\tau\,p)\,\int_0^\infty \left[t+\varrho\,|e^{i\,\vartheta}-e^{i\,\omega}|\right]^{-\tau\,p-2}\,dt.
\]
Thus from the definition of $\mathcal{W}_\varepsilon(u)$,  we obtain with simple manipulations
\[
\begin{split}
\mathcal{W}_\varepsilon(u)&\le C\,\int_0^\infty\int_\varepsilon^\infty \int_{[0,2\,\pi]\times[0,2\,\pi]} \frac{\left|u(\varrho\,e^{i\,\vartheta})-u\left((\varrho+t)\,e^{i\,\frac{\vartheta+\omega}{2}}\right)\right|^p}{\left(t+\varrho\,|e^{i\,\vartheta}-e^{i\,\omega}|\right)^{2+\tau\,p}}\,\varrho\,d\vartheta\,d\omega\,d\varrho\,dt.
\end{split}
\]
Observe that
\[
\left|\varrho\,e^{i\,\vartheta}-(\varrho+t)\,e^{i\,\frac{\vartheta+\omega}{2}}\right|\le t+\varrho\,\left|e^{i\,\vartheta}-e^{i\,\frac{\vartheta+\omega}{2}}\right|,
\]
and
\[
\left|e^{i\,\vartheta}-e^{i\,\frac{\omega+\vartheta}{2}}\right|\le C\,|e^{i\,\vartheta}-e^{i\,\omega}|.
\]
Hence,
\[
\begin{split}
\mathcal{W}_\varepsilon(u)&\le C\,\int_0^\infty\int_\varepsilon^\infty \int_{[0,2\,\pi]\times[0,2\,\pi]} \frac{\left|u(\varrho\,e^{i\,\vartheta})-u((\varrho+t)\,e^{i\,\frac{\vartheta+\omega}{2}})\right|^p}{\left|\varrho\,e^{i\,\vartheta}-(\varrho+t)\,e^{i\,\frac{\vartheta+\omega}{2}}\right|^{2+\tau\,p}}\,\varrho\,d\vartheta\,d\omega\,d\varrho\,dt\\
&\le 2\,\frac{C}{\varepsilon}\,\int_{[0,\infty)\times [0,\infty)} \int_{[0,2\,\pi]\times[0,2\,\pi]} \frac{|u(\varrho\,e^{i\,\vartheta})-u(s\,e^{i\,\psi})|^p}{|\varrho\,e^{i\,\vartheta}-s\,e^{i\,\psi}|^{2+\tau\,p}}\,\varrho\,s\,d\vartheta\,d\psi\,d\varrho\,ds,
\end{split}
\]
which completes the proof of \eqref{Ar}.
\end{proof}
The following result is a general fact for bounded $W^{1,2}$ functions in the plane. This is exactly the same as \cite[Lemma 5]{Santambrogio-Vespri}, we reproduce the proof for the reader's convenience. 
\begin{lm}\label{lemma-Santambrogio-Vespri-5}
Let \(\varphi\in W^{1,2}(B_R)\cap L^{\infty}(B_R)\) be a function such that \(0\leq \varphi \leq M\). Let us suppose that there exists \(0<\gamma<1\) such that
\[
\left|\left\{\varphi>\frac{3}{4}\, M\right\}\cap B_R\right| \geq \gamma \,|B_R|.
\]
If we set \(\delta=\sqrt{\gamma/2}\), one of the following two alternatives occur:
\begin{itemize}
\item[$(\mathbf{A}_1)$] either
\[
\int_{B_{R}\setminus B_{\delta R}}|\nabla \varphi|^2 \,dx \geq \frac{1}{512\,\pi}\,\gamma\, M^2;
\]
\vskip.2cm
\item[$(\mathbf{A}_2)$] or the following subset of $[\delta R,R]$
\[
\left\{s\in[\delta R,R]\, :\, \varphi\geq \frac{5}{8}\,M,\ \textrm{ $\mathcal{H}^1-$a.\,e. on } \partial B_s\right\},
\]
has positive measure.
\end{itemize}
\end{lm}
\begin{proof}
We first observe that thanks to the hypothesis we have
\[
\begin{split}
\left|\left\{\varphi>\frac{3}{4}\, M\right\}\cap (B_R\setminus B_{\delta R})\right|&=
\left|\left\{\varphi>\frac{3}{4}\, M\right\}\cap B_R\right|-
\left|\left\{\varphi>\frac{3}{4}\, M\right\}\cap B_{\delta R}\right|\\
&\ge \gamma\,|B_R|-|B_{\delta R}|=(\gamma-\delta^2)\,|B_R|.
\end{split}
\]
By definition of \(\delta\), we get
\[
\left|\left\{\varphi>\frac{3}{4}\, M\right\}\cap (B_R\setminus B_{\delta R})\right|\ge \frac{\gamma}{2}\,|B_R|.
\]
We define the set
\[
\mathcal{X}=\left\{s\in [\delta R,R]\, :\, \mathcal{H}^1\left(\left\{x\in\partial B_s\, :\, \varphi(x)\ge \frac{3}{4}\,M\right\}\right)>0\right\}.
\]
Then 
\[
\begin{split}
\frac{\gamma}{2}\,|B_R|\le \left|\left\{\varphi>\frac{3}{4}\, M\right\}\cap (B_R\setminus B_{\delta R})\right|&=\int_{\mathcal{X}} \int_{\partial B_s} 1_{\{\varphi>3/4\,M\}}\,d\mathcal{H}^1\,ds\\
&\le 2\,\pi\,\int_{\mathcal{X}} s\,ds\le 2\,\pi\, R\,|\mathcal{X}|.
\end{split}
\]
This in turn implies that
\[
|\mathcal{X}|\ge \frac{\gamma}{4}\,R.
\]
Let us now suppose that alternative $(\mathbf{A}_2)$ does not occur. This implies that 
\[
\mathcal{H}^1\left(\left\{x\in\partial B_s\,:\, \varphi(x)< \frac{5}{8}\,M\right\}\right)>0,\qquad \mbox{ for a.\,e. }s\in[\delta R,R].
\]
Thus for almost every $s\in \mathcal{X}$, we have
\[
\osc_{\partial B_s} \varphi\ge \frac{3}{4}\,M-\frac{5}{8}\,M=\frac{M}{8}.
\]
By observing that $\partial B_s$ is one-dimensional, we obtain
\[
\frac{M}{8}\le\osc_{\partial B_s} \varphi\le \int_{\partial B_s} |\nabla_\tau\varphi|\,d\mathcal{H}^1\le (2\,\pi\,R)^\frac{1}{2}\,\left(\int_{\partial B_s} |\nabla_\tau \varphi|^2\,d\mathcal{H}^1\right)^\frac{1}{2},
\]
where $\nabla_\tau$ denotes the tangential gradient (by using polar coordinates $x=\varrho\,e^{i\,\vartheta}$, this is nothing but the $\vartheta-$derivative). By taking the square in the previous estimate, integrating in $s\in\mathcal{X}$ and using the lower-bound on $|\mathcal{X}|$, we get
\[
\int_{B_{R}\setminus B_{\delta R}}|\nabla \varphi|^2 \,dx \geq \int_\mathcal{X} \int_{\partial B_s} |\nabla \varphi|^2\,d\mathcal{H}^1\ge \frac{M^2}{128}\,\frac{1}{\pi\,R}\,|\mathcal{X}|\ge \frac{M^2}{128}\,\frac{1}{\pi}\,\frac{\gamma}{4},
\]
which is alternative $(\mathbf{A}_1)$.
\end{proof}

\end{document}